\documentclass[a4paper]{amsart}
\usepackage{amsmath}
\usepackage{amssymb}
\usepackage{amsthm}
\usepackage[cp1250]{inputenc}
\usepackage[T1]{fontenc}
\usepackage{enumerate}
\usepackage{latexsym}
\usepackage{url}
\newtheorem{theorem}{Theorem}[section] 
\newtheorem{lemma}[theorem]{Lemma}     
\newtheorem{corollary}[theorem]{Corollary}
\newtheorem{proposition}[theorem]{Proposition}

\theoremstyle{definition}

\newtheorem{remark}[theorem]{Remark}

\newcommand{\n}{\mathbb{N}}
\newcommand{\z}{\mathbb{Z}}

\newcommand{\re}{\mathbb{R}}
\newcommand{\al}{\alpha}
\newcommand{\te}{\theta}
\newcommand{\la}{\lambda}
\newcommand{\g}{\gamma}
\newcommand{\ep}{\varepsilon}
\newcommand{\de}{\delta}

\begin{document}
\title[On special flows that are not isomorphic to their inverses]
{On special flows over IETs that are not isomorphic to their inverses}
\author[P. Berk \and K. Fr\k{a}czek ]{Przemysław Berk \and Krzysztof Fr\k{a}czek}

\address{Faculty of Mathematics and Computer Science\\ Nicolaus
Copernicus University\\ ul. Chopina 12/18\\ 87-100 Toru\'n, Poland} \email{zimowy@mat.umk.pl, fraczek@mat.umk.pl}
\date{\today}

\subjclass[2000]{37A10, 37E35}  \keywords{}
\thanks{Research partially supported by the Narodowe Centrum Nauki Grant
DEC-2011/03/B/ST1/00407.} \maketitle

\begin{abstract}
In this paper we give a criterion for a special flow to be not isomorphic to its inverse which  is a refine of a result in \cite{Fr-Ku-Le}.
We apply this criterion to special flows $T^f$ built over ergodic interval exchange transformations $T:[0,1)\to[0,1)$ (IETs) and under piecewise absolutely
continuous roof functions $f:[0,1)\to\re_+$. We show that for almost every IET $T$ if $f$ is absolutely continuous over exchanged intervals and has
non-zero sum of jumps then the special flow $T^f$ is not isomorphic to its inverse. The same conclusion is valid for a typical piecewise constant roof function.
\end{abstract}

\section{Introduction}\label{sec:intro}
The problem of isomorphism of  probability measure-preserving systems to their own inverse was already stated by Halmos-von Neumann in their seminal paper
\cite{Hal-vNeu}. In \cite{Hal-vNeu} the authors found a complete invariant for ergodic systems with discrete spectrum and then they applied it to prove
that any ergodic measure preserving transformation with pure point spectrum is isomorphic to its own inverse. Moreover, Halmos-von Neumann conjectured that
the same result is valid for an arbitrary measure preserving transformation. The first counter-example to this conjecture was given by Anzai in \cite{Anz},
it was so called Anzai skew product. Anzai counter-example gave the impetus for further research on the isomorphism problem of measure-preserving systems
to their inverse. As shown in \cite{dJun} (for automorphisms) and in \cite{DaRy} (for flows), the property of being non-isomorphic to the inverse is typical.
For a fairly  detailed introduction to the problem we refer also to \cite{Fr-Ku-Le}.

Recall that a measure-preserving flow $\{T_t\}_{t\in\re}$ on a standard probability Borel space $(X,\mathcal{B},\mu)$ is isomorphic to its inverse if there exists
measure-preserving automorphism $S:X\to X$, such that
\[
T_t\circ S=S\circ T_{-t}\text{ for all }t\in\re.
\]
For any ergodic measure-preserving automorphism $T:X\to X$ and  a positive integrable roof function  $f:X\to\re_+$
let us consider $X^f:=\{(x,r)\in X\times\re,0\le r<f(x)\}$. We deal with the \emph{special flow} $T^f$  on $X^f$ (see e.g.\ \cite{Co-Fo-Si}, Ch.11), that is the
flow which moves points in $X^f$ vertically upwards with unit speed and we identify the point $(x,f(x))$ with $(Tx,0)$.
If $T$ is an IET then the flows $T^f$ arise naturally as special representation of flows on compact surfaces.

In \cite{Fr-Ku-Le} the authors developed techniques to prove non-isomorphism of a flow $T^f$ to its inverse that based on
studying the weak closure of off-diagonal $3$-self-joinings. The idea of detecting non-isomorphism of a dynamical system and its inverse by studying the weak closure of off-diagonal $3$-self-joinings was already used by Ryzhikov in \cite{Ryz1}.
The tools developed in \cite{Fr-Ku-Le} were applied to the special flow built over irrational rotations and under piecewise absolutely continuous roof functions. The main aim of this paper is to extend the techniques of \cite{Fr-Ku-Le} to special flows over ergodic IETs. The paper is motivated by the desire to understand the problem of isomorphism of translation flows on translation surfaces to their own inverse. For the background material concerning translation surfaces, see \cite{Viana}. Recall that for every translation surface from any hyperelliptic component in the moduli space the vertical flow is isomorphic to its inverse. We conjecture that for a typical translation surface from any non-hyperelliptic component the vertical flow is not isomorphic to its inverse. We believe that the results of Section~\ref{sec:pieconst} lead to
the density of  translation surfaces for which the vertical flow is not isomorphic to its inverse in each such component. This gives a step toward the proof of the topological version of the conjecture.
The proof of the measure theoretical version seems to be beyond the reach of our approach.

In Section~\ref{sec:preli} we give a general background on special flows and joinings. We also recall so called off-diagonal joinings of higher rank, which serve as a main tool in latter calculations and constructions.

In Section~\ref{sec:lim} we state conditions under which a sequence of $3$-off-diagonal joinings converges weakly in the space $J_3(T^f)$ of all $3$-self-joinings and how does the limit looks like, see Theorem~\ref{glownyrach}. In \cite{Fr-Ku-Le} the authors give an explicit formula for the whole limit measure. Now under weaker assumptions the limit measure is controlled only partially. Nevertheless, it is enough for our purpose.  The proofs are based on ideas drawn from  \cite{Fr-Ku-Le} and \cite{FrLem3}.

In Section~\ref{sec:gencrit} (using results of Section~\ref{sec:lim}) we give a sufficient condition for the special flow built over partially rigid automorphisms to be not isomorphic to its inverse, see Theorem~\ref{glwynik}. This result reduces the problem
of non-isomorphism of $T^f$ and its inverse to establishing that a probability measure $\xi_*P$ on $\re$ does not satisfies a symmetry condition.

In the remainder of the paper we apply newly obtained criterion to special flows $T^f$ built over interval exchange transformations $T:[0,1)\to[0,1)$ and under piecewise
absolutely continuous roof functions $f:[0,1)\to\re_+$. The sum of jumps for such roof function will be denoted by $S(f)$.
In Section~\ref{sec:iets} we state some general background concerning IETs and Rauzy-Veech induction.

In  Section~\ref{sec:abscont} we give the proof of the following main result.
\begin{theorem}\label{thm:abscont}
Let $f:[0,1)\to\re_+$ be a piecewise absolutely continuous function with $\beta_1,\ldots,\beta_r\in [0,1)$ its
discontinuity points. For every irreducible permutation $\pi\in  S_d^0$ and almost every $\la\in \Lambda^d$ there exists $D_\la\subset[0,1)^r$ of full Lebesgue measure such that:
\begin{enumerate}
\item\label{ACn0} if $S(f)\neq 0$ and $f$ is absolutely continuous over intervals exchanged by $T_{\pi,\la}$, or
\item\label{AC0} if $S(f)=0$, $f$ has no jumps with opposite value, $r\ge 3$ and $(\beta_1,\ldots,\beta_r)\in D_\la$,
\end{enumerate}
then the special flow $T_{\pi,\la}^f$ built over the IET $T_{\pi,\la}$ and under the roof function $f$ is not isomorphic to its inverse.
\end{theorem}
To prove Theorem~\ref{thm:abscont} we require Lemma~\ref{general}, which consists of two parts, one concerning piecewise linear $f$ with constant non-zero slope and
other concerning piecewise constant $f$. They are proved separately in Sections~\ref{sec:overites}~and~\ref{sec:pieconst} respectively. Then a passage form piecewise linear
to piecewise absolutely continuous framework is delivered by Lemma~\ref{AClem}.

%

\section{Preliminaries}\label{sec:preli}
\subsection{Special flows}
Let $(X,\mathcal{B},\mu)$ be a standard probability Borel space. We will always assume that the measure $\mu$ is non-atomic. Suppose that
$T:(X,\mathcal{B},\mu)\to(X,\mathcal{B},\mu)$ is an ergodic automorphism. Let $\{V_n\}_{n\in\n}$ be a sequence of measurable subsets of $X$. We say that
$T$ is {\it rigid along a sequence} $\{V_n\}_{n\in\n}$ if there exists an increasing sequence of natural numbers $\{q_n\}_{n\in\n}$ such that
$\mu((T^{-q_n}A\triangle A)\cap V_n)\to 0$ for every measurable $A\subset X$. Then $\{q_n\}_{n\in\n}$ is called a {\it rigidity sequence} for $T$ along
$\{V_n\}_{n\in\n}$.

Assume that $f\in
L^1(X,\mathcal{B},\mu)$ is a strictly positive function. We will denote by $\mathcal B(\re)$ the standard
Borel $\sigma$-algebra on $\re$, while by $Leb$ we will denote the
Lebesgue measure on $\re$ or $[0,1)$ according to the context.  By
$\mathcal T^f=(T_t^f)_{t\in\re}$ we will denote the corresponding
special flow under $f$ acting on $(X^f,\mathcal B^f,\mu^f)$, where
$X^f=\{(x,r)\in X\times\re:\quad 0\le r<f(x)\}$ and $\mathcal B^f$
and $\mu^f$ are restrictions of $\mathcal B\otimes\mathcal B(\re)$
and $\mu\otimes Leb$ to $X^f$.
\begin{remark}
Since $\mu$ is non-atomic, the ergodicity of $T$ implies its aperiodicity. Therefore, the flow $T^f$ is ergodic and aperiodic, as well.
\end{remark}


\subsection{Self-joinings}
Let $\mathcal T=(T_t)_{t\in\re}$ be an ergodic flow on
$(X,\mathcal B,\mu)$. For every $k\ge 2$ by a {\it $k$-self-joining of
$\mathcal T$} we call any probability $(T_t\times\ldots \times
T_t)_{t\in\re}$-invariant measure $\la$ on $(X^k,\mathcal
B^{\otimes k})$ which projects as $\mu$ on each coordinate. We
will denote by $J_k(\mathcal T)$ the set of all $k$-self-joinings
of $\mathcal T$. We say that $k$-joining $\la$ is ergodic, if the
flow $(T_t\times\ldots \times T_t)_{t\in\re}$ is ergodic on
$(X^k,\la)$. For any $(T_t\times\ldots \times
T_t)_{t\in\re}$-invariant measure $\sigma$ we may consider its
unique {\it ergodic decomposition} $\sigma=\int_{\mathcal M_e}\nu
d\kappa(\nu)$, where $\mathcal M_e$ stands for the set of
$(T_t\times\ldots \times T_t)_{t\in\re}$-invariant ergodic
measures on $X^k$ and $\kappa$ is some probability measure on
$\mathcal M_e$.
\begin{remark}\label{rem:ergcomp}
If $\la=\int_{{\mathcal M}_e}\nu\, d\kappa(\nu)$ is the  ergodic
decomposition of a $k$-joining $\la$, then the set of measures
$\nu\in\mathcal{M}_e$ which are $k$-joinings is of full measure in
measure $\kappa$.
\end{remark}

Let $\{B_n:n\in\n\}$ be a countable family in $\mathcal B$ which is
dense in $\mathcal B$ for pseudo-metric
$d_\mu(A,B)=\mu(A\triangle B)$. Then on $J_k(\mathcal T)$ we may
consider a metric $\rho$ such that
\[
\rho(\la,\la ')=\sum_{n_1,\ldots,n_k\in\n}\frac{1}{2^{n_1+\ldots+n_k}}|\la(B_{n_1}\times\ldots\times B_{n_k})-\la '(B_{n_1}\times\ldots\times B_{n_k})|.
\]
The set $J_k(\mathcal T)$ endowed with the weak topology  derived
from this metric is compact. Moreover, a sequence of joinings
$(\la_n)_{n\in\n}$ is convergent to $\la$ in this metric whenever
$\la_n(A_1\times\ldots\times A_k)\to\la(A_1\times\ldots\times
A_k)$ for all $A_1,\ldots,A_k\in\mathcal B$.

Let $t_1,\ldots,t_{k-1}\in\re$. Then we may consider a
$k$-joining determined in following way
\[
\mu_{t_1,\ldots,t_{k-1}}(A_1\times\ldots\times A_{k-1}\times A_k)=\mu(T_{-t_1}A_1\cap\ldots\cap T_{-t_{k-1}}A_{k-1}\cap A_k),
\]
for $A_1,\ldots,A_k\in\mathcal B$. Each such joining is  called {\it
off-diagonal}. As the image of the measure $\mu$ via the map
$x\mapsto(T_{t_1}x,\ldots,T_{t_{k-1}}x,x)$, the joining
$\mu_{t_1,\ldots,t_{k-1}}$ is ergodic.
\begin{lemma}\label{suslin}
Suppose that $(T_t)_{t\in\re}$ is ergodic and aperiodic.  Then for any natural $k\ge 2$, the set $\mathcal{A}\subset J_k(\mathcal{T})$ of all
$k$-off-diagonal joinings is Borel in $J_k(\mathcal{T})$. Moreover ${h:\re^{k-1}\to\mathcal{A}}$ given by $h(t_1,\ldots,t_{k-1})=\mu_{t_1,\ldots,t_{k-1}}$
is a measurable isomorphism. $\Box$
\end{lemma}

We will denote by $\mathcal P(\re^{k})$ the space of all probability Borel measure on $\re^k$. For every $P\in\mathcal P(\re^{k-1})$  we will deal with an
integral $k$-self joining $\int_{\re^{k-1}}\mu_{t_1,\ldots,t_{k-1}}dP(t_1,\ldots,t_{k-1})$ such that
\[
\Big (\int_{\re^{k-1}}\mu_{t_1,\ldots,t_{k-1}}dP(t_1,\ldots,t_{k-1})\Big)(A):=\int_{\re^{k-1}}\mu_{t_1,\ldots,t_{k-1}}(A)dP(t_1,\ldots,t_{k-1}),
\]
for any $A\in\mathcal B^{\otimes k}$. In this paper  we will deal
with such integral joinings as partial limits of some sequences of
off-diagonal joinings. For relevant background material concerning joinings, see \cite{Glas}.

\section{Limit theorem for off-diagonal joinings}\label{sec:lim}
Let $(T^f_t)_{t\in\re}$ be an ergodic special flow on the space
$X^f$, where $T:(X,\mathcal B,\mu)\to(X,\mathcal B,\mu)$ is an
ergodic automorphism and $f\in L^2(X,\mathcal B,\mu)$ is a roof
function such that $f\ge c$ for some $c>0$. For every $n\in\z$ let
\[
f^{(n)}(x)=
\begin{cases}
f(x)+f(Tx)+\ldots+f(T^{n-1}x)& \text{ for }n>0  \\
0 & \text{ for }n=0  \\
-(f(T^{-1}x)+\ldots+f(T^{n}x))& \text{ for } n<0.
\end{cases}
\]
For any measurable subset $W\subset X$ with $\mu(W)>0$ we will denote by $\mu_W$ the conditional measure  given by $\mu_W(A)=\mu(A|W)$ for
$A\in\mathcal{B}$. Suppose that there exists a sequence $\{W_n\}_{n\in\n}$ of measurable subsets of $X$, increasing sequences $\{q_n\}_{n\in\n}$,
$\{q_n'\}_{n\in\n}$ of natural numbers and real sequences $\{a_n\}_{n\in\n}$, $\{a_n'\}_{n\in\n}$ such that following conditions are satisfied:

\begin{equation}\label{wl1}
\mu(W_n)\to\alpha \text{ with }0<\alpha\le 1,
\end{equation}
\begin{equation}\label{wl2}
\mu(W_n\triangle T^{-1}W_n)\to 0,
\end{equation}
\begin{equation}\label{wl3}
\{q_n\}_{n\in\n}\text{ is a rigidity sequence for $T$ along }\{W_n\}_{n\in\n},
\end{equation}
\begin{equation}\label{wl4}
\{q_n'\}_{n\in\n}\text{ is a rigidity sequence for $T$ along }\{W_n\}_{n\in\n},
\end{equation}
\begin{equation}\label{wl5}
\Big\{\int_{W_n}|f_n(x)|^2d\mu(x)\Big\}_{n\in\n}\text{ is bounded for }f_n=f^{(q_n)}-a_n,
\end{equation}
\begin{equation}\label{wl6}
\Big\{\int_{W_n}|f_n'(x)|^2d\mu(x)\Big\}_{n\in\n}\text{ is bounded for }f_n'=f^{(q_n')}-a_n',
\end{equation}
\begin{equation}\label{wl7}
(f_n',f_n)_*(\mu_{W_n})\to P\text{ weakly in }\mathcal{P}(\re^2).
\end{equation}
Similar conditions were stated as assumptions of Theorem 6 in \cite{FrLem}.
By the definition of weak convergence, we know that for  any
continuous bounded function $\phi:\re^2\to\re$, we have
\begin{equation}\label{eq1}
\int_{W_n}\phi(f_n'(x),f_n(x))d\mu(x)\to\alpha\int_{\re^2}\phi(t,u)dP(t,u).
\end{equation}

The purpose of this section is to prove Theorem~\ref{glownyrach} in which the limit
joining of the sequence $\{\mu^f_{a'_n,a_n}\}_{n\in\n}$ are gained.
The proof will need the following series of auxiliary lemmas.
\begin{lemma}\label{techB}
If $\{h_n\}_{n\in\n}$ and $\{g_n\}_{n\in\n}$ are bounded sequences
in $L^\infty(X,\mathcal B,\mu)$ such that $h_n\to 0$ in measure,
then $h_n\cdot g_n\to 0$ in $L^1$. $\Box$
\end{lemma}
\begin{lemma}\label{techA}
A sequence $\{q_n\}_{n\in\n}$ is  rigid for $T$ along
$\{W_n\}_{n\in\n}$ if and only if for every $f\in L^1(X,\mathcal
B,\mu)$ we have  $\chi_{W_n}(f\circ T^{q_n}-f)\to 0$ in measure.
\end{lemma}
\begin{proof}
Note that $\mu((T^{-q_n}A\triangle A)\cap W_n)\to 0$ is equivalent to
\[
\int_{W_n}|\chi_A\circ T^{q_n}-\chi_A|d\mu\to 0.
\]
By passing to simple functions and by density of simple  functions in $L^1$,
we have that $\chi_{W_n}(f\circ T^{q_n}-f)\to 0$ in $L^1$ for all
$f\in L^1(X,\mathcal B,\mu)$. By Markov's inequality,
$\chi_{W_n}(f\circ T^{q_n}-f)\to 0$ in measure.

Conversely, suppose that $\chi_{W_n}(f\circ T^{q_n}-f)\to 0$  in
measure for every $f\in L^1(X,\mathcal B,\mu)$. If $f$ is
additionally bounded then, by Lemma~\ref{techB},
$\chi_{W_n}(f\circ T^{q_n}-f)\to 0$ in $L^1$. Taking $f=\chi_A$ we
obtain $\chi_{W_n}|\chi_A\circ T^{q_n}-\chi_A|\to 0$ in $L^1$ for
every $A\in\mathcal{B}$. This gives the rigidity of the sequence
$\{q_n\}_{n\in\n}$ along $\{W_n\}_{n\in\n}$.
\end{proof}

\begin{lemma}\label{metr}
Suppose that $(X,\mathcal B,\mu)$ is endowed with a metric $d$
generating the $\sigma$-algebra~$\mathcal B$. If $\sup_{x\in
W_n}d(T^{q_n}x,x)\to 0$, then $\{q_n\}_{n\in\n}$ is a rigidity
sequence for $T$ along $\{W_n\}_{n\in\n}$ .
\end{lemma}
\begin{proof}
Let $h\in L^1(X,\mathcal B,\mu)$ and let $\ep>0$ and $a>0$ be arbitrary. Then, by Lusin's theorem, there exists a compact set $B_\ep\subset X$ such that
$\mu(B_\ep^c)<\frac{\ep}{2}$ and $h:B_\ep\to\re$ is uniformly continuous. Therefore, there exists $\de>0$ such that $d(x,y)<\de$ implies $|h(x)-h(y)|<a$
for all $x,y\in B_\ep$. By assumption, there exists $n_0\in\n$ such that
\[
n\ge n_0\text{ and }x\in W_n\Rightarrow d(x,T^{q_n}x)<\de.
\]
Hence,  $x\in W_n\cap B_\ep \cap T^{-q_n}B_\ep$ implies
$
|h(x)-h(T^{q_n}x)|<a
$
for $n>n_0$. Therefore
\[
\mu(\{x\in W_n:|h(x)-h(T^{q_n}x)|\ge a\})\le
\mu(W_n\cap (B_\ep^c \cup T^{-q_n}B_\ep^c))\le 2\mu(B_\ep^c)<\ep.
\]
As $\ep>0$ and $a>0$ were arbitrary and using Lemma~\ref{techA}, we complete the proof.
\end{proof}

\begin{theorem} \label{tech}
Suppose that \eqref{wl1}-\eqref{wl7} hold. Let $h$, $h':X\to \re$
be  measurable functions. Let $g\in L^\infty(X,\mathcal{B},\mu)$
and $\phi:\re^2\to \re$ be bounded and uniformly continuous. Then
\begin{equation}\label{gran}
\begin{split}
\int_{W_n}&\phi(f_n'(x)+h'(x),f_n(x)+h(x))g(x)\,d\mu(x)\\&\to
\alpha \int_{X}\int_{\re^2}\phi(t+h'(x),u+h(x))g(x)\,dP(t,u)\,d\mu(x).
\end{split}
\end{equation}
\end{theorem}
\begin{proof}
We  divide the proof into steps according to the complexity of the functions $h$ and $h'$.\\

\textbf{Step 1.} Assume that $h=h'=0$. If $g$ is constant then
\eqref{gran}  follows from \eqref{eq1} directly. Hence we can
assume that $g\in L^1_0(X,\mathcal{B},\mu)$, i.e.\ has zero mean.
By the proof of von Neumann ergodic theorem, coboundaries, i.e.\
functions of the form $g=\xi-\xi\circ T$ with $\xi\in
L^2(X,\mathcal{B},\mu)$ are dense in $L^1_0(X,\mathcal{B},\mu)$.
Therefore, it suffices  to consider $g=\xi-\xi\circ T$ for $\xi\in
L^\infty(X,\mathcal{B},\mu)$, as they are also dense in
$L^1_0(X,\mathcal{B},\mu)$. Note that the RHS of \eqref{gran} is
equal to
\[
\alpha \int_{\re^2}\phi(t,u)dP(t,u)\int_{X}g(x)\,d\mu(x)=0,
\]
whenever $g\in L^1_0(X,\mathcal{B},\mu)$. As $g=\xi-\xi\circ T$, we need to prove that
\begin{equation}\label{calkakobrzeg}
\Big|\int_{W_n}\phi(f_n'(x),f_n(x))\xi(x)\,\,d\mu(x)-\int_{W_n}\phi(f_n'(x),f_n(x))\xi(Tx)\,\,d\mu(x)\Big|\to 0.
\end{equation}
However, by the $T$-invariance of $\mu$, we have
\[
\begin{split}
\Big|&\int_{W_n}\phi(f_n'(x),f_n(x))\xi(x)\,d\mu(x)-\int_{W_n}\phi(f_n'(x),f_n(x))\xi(Tx)\,d\mu(x)\Big|\\&=
\Big|\int_{T^{-1}W_n}\phi(f_n'(Tx),f_n(Tx))\xi(Tx)\,d\mu(x)-\int_{W_n}\phi(f_n'(x),f_n(x))\xi(Tx)\,d\mu(x)\Big|\\&\le
\int_{W_n}|\phi(f_n'(Tx),f_n(Tx))-\phi(f_n'(x),f_n(x))||\xi(Tx)|\,d\mu(x)\\&\quad+\|\phi\|_{\infty}\int_{T^{-1}W_n\triangle W_n}|\xi(Tx)|\,d\mu(x).
\end{split}
\]
By \eqref{wl2}, we have $\mu(T^{-1}W_n\triangle W_n)\to 0$.  Thus,
$\int_{T^{-1}W_n\triangle W_n}|\xi(Tx)|\,d\mu(x)\to 0$ as
$n\to\infty$.

Now we use the uniform continuity of $\phi$. By the definition of $f_n$ and $f_n'$, we have
\[
(f_n'(Tx),f_n(Tx))-(f_n'(x),f_n(x))=(f(T^{q_n'}x)-f(x),f(T^{q_n}x)-f(x)).
\]
By Lemma \ref{techA}, we have that
\[
\chi_{W_n}(x)(f(T^{q_n'}x)-f(x))\to 0\text{  and  }\chi_{W_n}(x)(f(T^{q_n}x)-f(x))\to 0
\]
in measure and thus
\[
\chi_{W_n}(x)\Big((f_n'(Tx),f_n(Tx))-(f_n'(x),f_n(x))\Big)\to 0
\]
in measure. Since $\phi$ is uniformly continuous, we also have
\[
\chi_{W_n}\Big(\phi(f_n'\circ T,f_n\circ T)-\phi(f_n',f_n)\Big)\to 0
\]
in measure. By Lemma~\ref{techB}, we get that
\[
\int_{W_n}|\phi(f_n'(Tx),f_n(Tx))-\phi(f_n'(x),f_n(x))||\xi(Tx)|\,d\mu(x)\to 0.
\]
This concludes the proof of \eqref{calkakobrzeg}, which also completes the proof of \eqref{gran} for $h=h'=0$.\\

\textbf{Step 2.} Now let $h'=\sum_{i=1}^{k}h'_i\chi_{A_i}$  and
$h=\sum_{j=1}^lh_j\chi_{B_j}$  be simple functions, where $A_i$
and $B_j$ for $i=1,\ldots,k$ and $j=1,\ldots,l$ make two
measurable disjoint partitions of $X$. Then
\[
\begin{split}
\int_{W_n}&\phi(f_n'(x)+h'(x),f_n(x)+h(x))g(x)\,d\mu(x)\\&
=\sum_{i=1}^{k}\sum_{j=1}^l\int_{W_n}\phi(f_n'(x)+h_i',f_n(x)+h_j)g(x)\chi_{A_i}(x)\chi_{B_j}(x)\,d\mu(x)\\&\to
\sum_{i=1}^{k}\sum_{j=1}^l\alpha\int_X\int_{\re^2}\phi(t+h_i',u+h_j)g(x)\chi_{A_i}(x)\chi_{B_j}(x)dP(t,u)\,d\mu(x)\\&
=\alpha\int_X\int_{\re^2}\phi(t+h'(x),u+h(x))g(x)dP(t,u)\,d\mu(x),
\end{split}
\]
where the convergence follows from the first step of the proof applied to functions
$(t,u)\mapsto\phi(t+h_i',u+h_j)$. This gives \eqref{gran} whenever $h$ and $h'$
are simple functions.\\

\textbf{Step 3.} All we need to show now is that for  arbitrary
measurable functions $h$ and $h'$, we can find sequences $\{h_m\}_{m\in\n}$, $\{h_m'\}_{m\in\n}$ of simple
functions such that
\begin{align*}
\int_{W_n}&\phi(f_n'(x)+h_m'(x),f_n(x)+h_m(x))g(x)\,d\mu(x)\\&
\to\int_{W_n}\phi(f_n'(x)+h'(x),f_n(x)+h(x))g(x)\,d\mu(x)
\end{align*}
and
\begin{align*}
\int_{X}&\int_{\re^2}\phi(t+h_m'(x),u+h_m(x))g(x)dP(t,u)\,d\mu(x)\\&
\to\int_{X}\int_{\re^2}\phi(t+h'(x),u+h(x))g(x)dP(t,u)\,d\mu(x),
\end{align*}
as $m\to+\infty$.
Take $h_m$ and $h_m'$ simple, such that $h_m\to h$ and  $h_m'\to
h'$ in measure. Then, by the uniform continuity of $\phi$, we obtain
that
\[
\phi(f_n'(x)+h_m'(x),f_n(x)+h_m(x))-\phi(f_n'(x)+h'(x),f_n(x)+h(x))\to 0,
\]
in measure on $(X,\mu)$ and
\[
\phi(t+h_m'(x),u+h_m(x))-\phi(t+h'(x),u+h(x))\to 0
\]
in measure on $(\re^2\times X,P\otimes\mu)$. By  Lemma
\ref{techB} and Step 2, this completes the proof of the whole
theorem.
\end{proof}

\begin{lemma}\label{tech2}
Let the assumptions \eqref{wl1}-\eqref{wl7} hold.  Furthermore,
assume that $g,\xi,\xi '\in L^\infty(X,\mathcal{B},\mu)$. Then
\[
\begin{split}
\int_{W_n}&\phi(f_n'(x)+h'(x),f_n(x)+h(x))\,g(x)\,\xi(T^{q_n}x)\,\xi '(T^{q_n'}x)\,\,d\mu(x)\\&\to
\alpha \int_{X}\int_{\re^2}\phi(t+h'(x),u+h(x))\,g(x)\,\xi(x)\,\xi '(x)\,dP(t,u)\,\,d\mu(x).
\end{split}
\]
\end{lemma}
\begin{proof}
By Lemma \ref{techA},
\[
\chi_{W_n}(x)(\xi(x)-\xi(T^{q_n}x))\to 0
\text{  and  }
\chi_{W_n}(x)(\xi '(x)-\xi '(T^{q_n'}x))\to 0,
\]
in measure. Then, by Lemma \ref{techB}, it
follows that
\begin{align*}
\Big|&\int_{W_n}\phi(f_n'(x)+h'(x),f_n(x)+h(x))g(x)\xi(T^{q_n}x)\xi '(T^{q_n'}x)\,d\mu(x)\\
&\qquad-\int_{W_n}\phi(f_n'(x)+h'(x),f_n(x)+h(x))g(x)\xi(x)\xi '(T^{q_n'}x)\,d\mu(x)\Big|\\&
\le \int_X|\phi(f_n'(x)+h'(x),f_n(x)+h(x))||g(x)||\xi '(T^{q_n'}x)||\chi_{W_n}(x)(\xi(T^{q_n}x)-\xi(x))|\,d\mu(x)\to 0,
\end{align*}
and
\begin{align*}
\Big|&\int_{W_n}\phi(f_n'(x)+h'(x),f_n(x)+h(x))g(x)\xi(x)\xi '(T^{q_n'}x)\,d\mu(x)\\&
\qquad -\int_{W_n}\phi(f_n'(x)+h'(x),f_n(x)+h(x))g(x)\xi(x)\xi '(x)\,d\mu(x)\Big|\\&\le
\int_X|\phi(f_n'(x)+h'(x),f_n(x)+h(x))||g(x)||\xi(x)||\chi_{W_n}(x)(\xi '(T^{q_n'}x)-\xi '(x))|\,d\mu(x)\to 0.
\end{align*}
Hence to conclude the proof of the lemma, we need to show that
\begin{align*}
&\int_{W_n}\phi(f_n'(x)+h'(x),f_n(x)+h(x))g(x)\xi(x)\xi '(x)\,d\mu(x)\\&\to
\alpha \int_{X}\int_{\re^2}\phi(t+h'(x),u+h(x))g(x)\xi(x)\xi '(x)dP(t,u)\,d\mu(x).
\end{align*}
However, that is a direct consequence of Theorem~\ref{tech}, taking $g(x)\xi(x)\xi '(x)$ in place of $g(x)$.
\end{proof}

The following auxiliary lemma is well-known and we state it without any proof.
\begin{lemma}\label{dodfun}
Let $g_n:X\to\re^m$, $n\in\n$ be measurable maps such that $(g_n)_*\mu\to P$ weakly in
$\mathcal{P}(\re^m)$ and  let $h_n:X\to\re^n$, $n\in\n$ be measurable maps such that $h_n\to 0$
in measure. Then $(g_n+h_n)_*\mu\to P$ weakly in
$\mathcal{P}(\re^m)$. $\Box$
\end{lemma}

Let $T:(X,\mathcal{B},\mu)\to(X,\mathcal{B},\mu)$ be an ergodic automorphism and let $f:X\to\re$
be square integrable such that $f\geq c>0$. Denote by $T_{-f}:X\times\re\to X\times\re$ the skew product
$T_{-f}(x,r)=(Tx,r-f(x))$. Then for every $n\in\z$ we have $T_{-f}^n(x,r)=(T^nx,r-f^{(n)}(x))$.
Denote by $(\sigma_t)_{t\in\re}$ the flow on $X\times\re$ defined by $\sigma_t(x,r)=(x,r+t)$.

The following lemma follows directly from Lemma~3.2 in \cite{Fr-Ku-Le}.
\begin{lemma}\label{sumamiar}
For all $t,s\in\re$ and all measurable sets $A,B,C\subset X^f$ we have
\[
\mu^f\big(T^f_tA\cap T^f_sB\cap C\big)=\sum_{k,l\in\z}\mu\otimes Leb\big((T_{-f})^k\sigma_tA\cap(T_{-f})^l\sigma_sB\cap C\big).
\]
Moreover, the sets that appear on the RHS of the above equation are pairwise disjoint.
\end{lemma}

\begin{lemma}[see Lemma~3.4 in \cite{Fr-Ku-Le}]\label{miaranacalke}
Suppose that $A=A_1\times A_2$, $B=B_1\times B_2$, $C=C_1\times
C_2$  are measurable rectangles in $X\times \re$. Then
\begin{multline*}
\mu\otimes Leb \big(((T_{-f})^{k_1}A)\cap((T_{-f})^{k_2}B)\cap C\big)\\=
\int_{(T^{k_1}A_1)\cap(T^{k_2}B_1)\cap C_1} Leb \Big(\big(A_2+f^{(-k_1)}(x)\big)\cap\big(B_2+f^{(-k_2)}(X)\big)\cap C_2\Big)\,d\mu(x).
\end{multline*}
\end{lemma}

The following theorem is inspired by Theorem~6 in \cite{FrLem} and Proposition~3.7 in \cite{Fr-Ku-Le}.
More precisely, this is a version of Proposition~3.7 in \cite{Fr-Ku-Le} in which the rigidity assumption for $T$
is replaced by (a weaker assumption) rigidity along a sequence of sets. Although the thesis of the present version is
also weaker, the theorem can be used to prove non-isomorphism to the inverse for special flows built over some not necessarily rigid  automorphisms $T$.
Recall that Proposition~3.7 in \cite{Fr-Ku-Le} works for special flows built over irrational rotations.
New version of this theorem applies to special flows over ergodic interval exchange transformations.
\begin{theorem}\label{glownyrach}
Suppose that \eqref{wl1}-\eqref{wl7} hold. Then
\begin{equation*}
\mu_{a_n',a_n}^f\to\rho=\alpha\int_{\re}\mu_{-t,-u}^fdP(t,u)+(1-\alpha)\nu\quad \text{ in }J_3(\mathcal{T}^f),
\end{equation*}
where $\nu\in J_3(\mathcal{T}^f)$.
\end{theorem}
\begin{proof}
By the compactness of $J_3(\mathcal{T}^f)$, and by passing to a subsequence, if necessary,
we have $\mu_{a_n',a_n}^f\to\rho$ in $J_3(\mathcal{T}^f)$. First we will prove
that for measurable rectangles in $X^f$
\[
A=A_1\times A_2,\quad B=B_1\times B_2,\quad C=C_1\times C_2,
\]
with $A_2,B_2,C_2\subset\re$ bounded, we have
\begin{equation}\label{czesc}
\begin{split}
\mu^f\Big(T^f_{-a_n'}&\big(A\cap(T^{q_n'}W_n\times\re)\big)\cap T^f_{-a_n}\big(B\cap(T^{q_n}W_n\times\re)\big)\cap C\Big)\\&
 \to \alpha\int_{\re^2}\mu^f\Big(T^f_tA\cap T^f_uB\cap C\Big)\,dP(t,u).
\end{split}
\end{equation}
By Lemma \ref{sumamiar},
\begin{multline*}
\mu^f\Big(T^f_{-a_n'}\big((A_1\cap T^{q_n'}W_n)\times
A_2\big)\cap T^f_{-a_n}\big((B_1\cap T^{q_n}W_n)\times B_2\big)\cap
(C_1\times C_2)\Big)\\= \sum_{k\in\z}\sum_{l\in\z}\mu\otimes Leb
\Big((T_{-f})^{-k}(T_{-f})^{-q_n'}\sigma_{-a_n'}\big((A_1\cap
T^{q_n'}W_n)\times A_2\big)\\
\cap(T_{-f})^{-l}(T_{-f})^{-q_n}\sigma_{-a_n}\big((B_1\cap
T^{q_n}W_n)\times B_2\big)\cap (C_1\times C_2)\Big).
\end{multline*}
Moreover, in view of Lemma \ref{miaranacalke},
\[
\begin{split}
a_{k,l}^n:=&\mu\otimes Leb
\Big((T_{-f})^{-k}(T_{-f})^{-q_n'}\sigma_{-a_n'}\big((A_1\cap
T^{q_n'}W_n)\times A_2\big)\\&\qquad
\cap(T_{-f})^{-l}(T_{-f})^{-q_n}\sigma_{-a_n}\big((B_1\cap
T^{q_n}W_n)\times B_2\big)\cap (C_1\times C_2)\Big)\\=& \int_{U_n} Leb
\Big(\big(A_2-a_n'+f^{(q_n'+k)}(x)\big)\cap\big(B_2-a_n+f^{(q_n+l)}(x)\big)\cap
C_2\Big)\,d\mu(x),
\end{split}
\]
where \begin{eqnarray*}
U_n&:=&T^{-q_n'-k}\big(A_1\cap T^{q_n'}W_n\big)\cap T^{-q_n-l}\big(B_1\cap T^{q_n}W_n\big)\cap C_1\\
&=&T^{-q_n'-k}A_1\cap T^{-k}W_n\cap T^{-q_n-l}B_1\cap T^{-l}W_n\cap C_1.
\end{eqnarray*}
Fix $l\in\z$. Using Lemma \ref{sumamiar} and then Lemma \ref{miaranacalke} for $A:=X\times\re$ we obtain
\[
\begin{split}
\sum_{k\in\z}a_{k,l}^n&\le\mu\otimes Leb
\Big((T_{-f})^{-l}(T_{-f})^{-q_n}\sigma_{-a_n}\big((B_1\cap T^{q_n}W_n)\times B_2\big)\cap (C_1\times C_2)\Big)\\
&\le
\int_{T^{-l}W_n} Leb \Big(\big(B_2-a_n+f^{(l+q_n)}(x)\big)\cap C_2\Big)\,d\mu(x)\\&=
\int_{T^{-l}W_n} Leb \Big(\big(B_2+f_n(T^lx)+f^{(l)}(x)\big)\cap C_2\Big)\,d\mu(x),
\end{split}
\]
where in the last equality we used the fact that
\[
f^{(l+q_n)}(x)-a_n=f^{(l)}(x)+f^{(q_n)}(T^lx)-a_n=f^{(l)}(x)+f_n(T^lx).
\]
Let now $s=diam(B_2\cup C_2)$ and $V_n=\big\{x\in T^{-l}W_n:|f_n(T^lx)+f^{(l)}(x)|\le s\big\}$.  Then
\begin{equation*}
\begin{split}
\int_{T^{-l}W_n} &Leb \Big(\big(B_2+f_n(T^lx)+f^{(l)}(x)\big)\cap C_2\Big)\,d\mu(x)\\&=
\int_{V_n} Leb \Big(\big(B_2+f_n(T^lx)+f^{(l)}(x)\big)\cap C_2\Big)\,d\mu(x)\\&\le
s\mu(V_n)\le s\mu\big(\big\{x\in W_n:|f_n(x)|\ge c|l|-s\big\}\big)\le s\frac{D}{(c|l|-s)^2},
\end{split}
\end{equation*}
where $D=\sup_{n\in\n}\int_{W_n}|f_n(x)|^2\,d\mu(x)$ and the last inequality  follows from Chebyshev's inequality.
Let $D_l=s\frac{D}{(c|l|-s)^2}$ for $l\in\z$ with $|l|>\frac{s}{c}$ and let $D_l=\mu^f(X^f)$ otherwise. Then $\sum_{k\in\z}a_{k,l}^n\le
D_l$ and $\sum_{l\in\z}D_l<\infty$. Similarly, we can find a sequence $\{D_k'\}_{k\in\z}$ such that $\sum_{k\in\z}D_k'<\infty$ and
$\sum_{l\in\z}a_{k,l}^n\le D_k'$. Hence, for every $\ep>0$ there exists $M>0$ such that
\[
\sum_{|k|\ge M}\sum_{l\in\z}a_{k,l}^n\leq\sum_{|k|\ge M} D_k'<\frac{\ep}{8}
\quad\text{
and }\quad
\sum_{|l|\ge M}\sum_{k\in\z}a_{k,l}^n\leq\sum_{|l|\ge M} D_l<\frac{\ep}{8}.
\]
 Thus
\begin{equation}\label{1ep}
\sum_{\max(|k|,|l|)\ge M}a_{k,l}^n\leq\sum_{|k|\ge M}\sum_{l\in\z}a_{k,l}^n+\sum_{|l|\ge M}\sum_{k\in\z}a_{k,l}^n<\frac{\ep}{4}.
\end{equation}
We are now going to prove that for every pair $(k,l)\in\z^2$ sequence
\begin{multline*}
a_{k,l}^n=\int_{T^{-k}W_n\cap T^{-l}W_n} \chi_{C_1}(x)\chi_{T^{-k}A_1}(T^{q_n'}x)\chi_{T^{-l}B_1}(T^{q_n}x)\\
Leb \Big(\big(A_2+f_n'(x)+f^{(k)}(T^{q_n'}x)\big)\cap \big(B_2+f_n(x)+f^{(l)}(T^{q_n}x)\big)\cap C_2\Big)\,d\mu(x)
\end{multline*}
converges. Since, by assumption \eqref{wl2}, $\mu\big((T^{-k}W_n\cap
T^{-l}W_n)\triangle W_n\big)\to 0$, it is enough to check the
convergence of the sequence
\begin{multline*}
b_{k,l}^n:=\int_{W_n} Leb \Big(\big(A_2+f_n'(x)+f^{(k)}(T^{q_n'}x)\big)\cap
\big(B_2+f_n(x)+f^{(l)}(T^{q_n}x)\big)\cap C_2\big)\\
\chi_{C_1}(x)\chi_{T^{-k}A_1}(T^{q_n'}x)\chi_{T^{-l}B_1}(T^{q_n}x)\,d\mu(x).
\end{multline*}
Let $F_n',F_n:X\to\re$ be given by
\[
F_n'(x)=f_n'(x)+f^{(k)}(T^{q_n'}x)-f^{(k)}(x) \text{  and  }
F_n(x)=f_n(x)+f^{(l)}(T^{q_n}x)-f^{(l)}(x).
\]
Then
\begin{multline*}
b_{k,l}^n=\int_{W_n} Leb \Big(\big(A_2+F_n'(x)+f^{(k)}(x)\big)\cap \big(B_2+F_n(x)+f^{(l)}(x)\big)\cap C_2\Big)\\
\chi_{C_1}(x)\chi_{T^{-k}A_1}(T^{q_n'}x)\chi_{T^{-l}B_1}(T^{q_n}x)\,d\mu(x).
\end{multline*}
By Lemma \ref{techA},
\[\chi_{W_n}(x)(f^{(k)}(T^{q_n}x)-f^{(k)}(x))\to 0\text{  and }
\chi_{W_n}(x)(f^{(l)}(T^{q_n'}x)-f^{(l)}(x))\to 0\]
in
measure. Since $(f_n',f_n)_*(\mu_{W_n})\to P$,
by Lemma~\ref{dodfun}, it implies
\[
(F_n',F_n)_*(\mu_{W_n})\to P \text{ weakly in }\mathcal{P}(\re^2).
\]
Let us apply Lemma~\ref{tech2} to
\[
\begin{split}
 \phi(t,u)&:= Leb \big((A_2+t)\cap (B_2+u)\cap C_2\big),\quad
 (h',h):=(f^{(k)},f^{(l)}),\quad
 (f_n',f_n):=(F_n',F_n),
\\ g&:=\chi_{C_1},
\quad \xi':=\chi_{T^{-k}A_1},
\quad \xi:=\chi_{T^{-l}B_1}.
\end{split}
\]
We obtain that
\begin{multline*}
b_{k,l}^n\to c_{k,l}:=\alpha\int_X\int_{\re^2} Leb \Big(\big(A_2+t+f^{(k)}(x)\big)\cap \big(B_2+u +f^{(l)}(x)\big)\cap C_2\Big)\\
 \chi_{C_1}(x)\chi_{T^{-k}A_1}(x)\chi_{T^{-l}B_1}(x)\,dP(t,u)\,d\mu(x).
\end{multline*}
By Fubini's theorem and Lemma~\ref{miaranacalke}, $c_{k,l}$ is equal to
\begin{multline*}
\alpha\int_{\re^2}\int_{T^{-k}A_1\cap T^{-l}B_1\cap C_1} Leb \Big(\big(A_2+t+f^{(k)}(x)\big)\cap \big(B_2+u+f^{(l)}(x)\big)\cap C_2\Big)
\,d\mu(x)dP(t,u)\\=
\alpha\int_{\re^2}(\mu\otimes Leb) \Big((T_{-f})^{-k}\sigma_t(A_1\times A_2)
\cap(T_{-f})^{-l}\sigma_u(B_1\times B_2)\cap (C_1\times C_2)\Big)\,dP(t,u).
\end{multline*}
By  Lemma \ref{sumamiar},
\begin{equation}\label{eq:sumce}
\sum_{k\in\z}\sum_{l\in\z}c_{k,l}=\alpha\int_{\re^2}\mu^f\big(T^f_tA\cap T^f_uB\cap C\big)dP(t,u)<\infty.
\end{equation}
Hence, by enlarging $M$ if necessary, we have
\begin{equation}\label{2ep}
\sum_{\max\{|k|,|l|\}>M}c_{k,l}<\frac{\ep}{4}.
\end{equation}
Since $a_{k,l}^n\to c_{k,l}$ for all $k,l\in\z$, we can choose $N>0$ so that
for all  $n\ge N$ and $k,l\in\z$ with $\max\{|k|,|l|\}\le
M$ we have
\[
|a_{k,l}^n-c_{k,l}|<\frac{\ep}{2(2M+1)^2}.
\]
In view of \eqref{1ep} and \eqref{2ep}, it follows that
\begin{align*}
\Big|\sum_{k,l\in\z}&a_{k,l}^n-\sum_{k,l\in\z}c_{k,l}\Big|
\le \sum_{\max\{|k|,|l|\}>M}a_{k,l}^n+\sum_{\max\{|k|,|l|\}>M}c_{k,l}+\sum_{\max\{|k|,|l|\}
\le M}|a_{k,l}^n-c_{k,l}|\le \ep.
\end{align*}
Therefore, $\sum_{k,l\in\z}a_{k,l}^n\to\sum_{k,l\in\z}c_{k,l}$ and,
by \eqref{eq:sumce}, this proves the convergence
\eqref{czesc}.

Since $\rho$ is the weak limit of  $\mu^f_{a_n',a_n}$,  by \eqref{czesc}, we get
\[
\rho(A\times B\times C)\ge\al\int_{\re^2}\mu^f_{-t,-u}(A\times B\times C)dP(t,u),
\]
for  measurable $A,B,C\subset X^f$. The proof is concluded by using the uniqueness of the ergodic decomposition of joinings.
%
\end{proof}

\section{Non-reversibility criteria for special flows}\label{sec:gencrit}
Let $T:( X,\mathcal B,\mu)\to( X,\mathcal B,\mu)$ be an ergodic automorphism. Let $(T^f_t)_{t\in\re}$ be the special flow built over $T$ and under a square
integrable roof function $f:X\to\re_+$. Suppose that conditions \eqref{wl1}-\eqref{wl7} are satisfied for a sequence $\{W_n\}_{n\in\n}$ of measurable sets,
integer sequences $\{q_n\}_{n\in\n}$, $\{q_n'\}_{n\in\n}$, real sequences $\{a_n\}_{n\in\n}$, $\{a_n'\}_{n\in\n}$, a real number $0<\al\le 1$ and a measure
$P\in\mathcal P(\re^2)$. Recall that
\[
(f_n',f_n)_*(\mu_{W_n})\to P\text{ weakly in }\mathcal{P}(\re^2).
\]
Moreover, suppose that $q_n'=2q_n$ and $a_n'=2a_n$.
Then, by Theorem  \ref{glownyrach}, we obtain the following.
\begin{corollary}
Let $P$, $\{a_n\}_{n\in\n}$, $\al$ be defined as above. Then
\begin{equation}\label{rdzen}
\mu_{2a_n,a_n}^f\to\alpha\int_{\re}\mu_{-t,-u}^fdP(t,u)+(1-\alpha)\nu,
\end{equation}
where $\nu$ is a 3-joining of the special flow $T^f$.
\end{corollary}
We are going to show a necessary condition for $T^f$  to be
isomorphic to its inverse. Assume then, that there exists a
measure-preserving isomorphism $S:X^f\to X^f$ such that
\begin{equation}\label{eq:revers}
ST^f_t=T^f_{-t}S\quad\text{ for }\quad t\in\re.
\end{equation}
\begin{lemma}\label{rownoscmiar}
Suppose that \eqref{rdzen} and \eqref{eq:revers} are satisfied. Let us consider  the map
$\te:\re^2\to\re^2$ given by $\te(t,u)=(t,t-u)$. Then
\begin{equation}\label{czesciowy}
\alpha\int_{\re^2}\mu_{t,u}^fdP(t,u)+(1-\alpha)\rho_1=\alpha\int_{\re^2}\mu_{t,u}^f\,d\te_*P(t,u)+(1-\alpha)\rho_2,
\end{equation}
where $\rho_1,\rho_2\in J_3(T^f)$.
\end{lemma}
\begin{proof}
Denote by $S_*$ the map which pushes forward  any measure on
$(X^f)^3$ by the map $S\times S\times S:(X^f)^3\to(X^f)^3$. Note
that $S_*$ maps 3-joinings into 3-joinings.
Additionally, $S_*:J_3(T^f)\to J_3(T^f)$ is continuous and affine. It is also easy to verify that $S_*\mu_{t,u}^f=\mu_{-t,-u}^f$.
%
By the continuity and affinity of $S_*$ and by \eqref{rdzen}, we have
\begin{equation}\label{cz1}
\begin{split}
S_*\mu_{2a_n,a_n}^f\to&\ \alpha\int_{\re^2}S_*\mu_{-t,-u}^fdP(t,u)+(1-\alpha)S_*\nu=
\alpha\int_{\re^2}\mu_{t,u}^fdP(t,u)+(1-\alpha)S_*\nu.
\end{split}
\end{equation}
However, by a direct computation, using again \eqref{rdzen} we obtain
\begin{equation}\label{cz2}
\begin{split}
S_*&\mu_{2a_n,a_n}^f(A\times B\times C)=\mu_{-2a_n,-a_n}^f(A\times B\times C)=
\mu^f(T^f_{2a_n}A\cap T^f_{a_n}B\cap C)\\&=\mu^f(A\cap T^f_{-a_n}B\cap T^f_{-2a_n}C)=
\mu^f_{2a_n,a_n}(C\times B\times A)\\&\to \alpha\int_{\re^2}\mu_{-t,-u}^f(C\times B\times A)
dP(t,u)+(1-\alpha)\nu(C\times B\times A)\\&=\alpha\int_{\re}\mu^f(A\cap T^f_{u}B\cap T^f_{t}C)dP(t,u)+(1-\alpha)\nu(C\times B\times A)\\&=
\alpha\int_{\re^2}\mu^f(T^f_{-t}A\cap T^f_{u-t}B\cap C)dP(t,u)+(1-\alpha)\nu(C\times B\times A)\\&=\alpha\int_{\re^2}
\mu^f_{t,t-u}(A\times B\times C)dP(t,u)+(1-\alpha)\nu(C\times B\times A)\\&=
\alpha\int_{\re^2}\mu^f_{t,u}(A\times B\times C)\,d\te_*P(t,u)+(1-\alpha)\nu(C\times B\times A).
\end{split}
\end{equation}
Now \eqref{czesciowy} follows from \eqref{cz1} and \eqref{cz2}.
\end{proof}

\begin{lemma}\label{rownosc}
Let $(T_t^f)_{t\in\re}$ be a special flow on $X^f$ isomorphic to its inverse.  Assume that \eqref{rdzen} is satisfied. Then
\begin{equation}\label{glownemiary}
\alpha P+\beta(1-\alpha)\eta_1=\alpha \te_*P+\beta(1-\alpha)\eta_2,
\end{equation}
for some measures $\eta_1,\eta_2\in\mathcal{P}(\re^2)$ and $0\leq\beta\leq 1$.
\end{lemma}
\begin{proof}
In view of Lemma~\ref{rownoscmiar}, \eqref{czesciowy} holds.
Let $\rho_1$ and $\rho_2$ be measures given by \eqref{czesciowy}. Let
\begin{equation}\label{eq:ergcompp}
\begin{split}
&\rho_1=\int_{J_3(T^f)}\eta \,d\kappa_1(\eta)=\beta\int_\mathcal{A}\eta
\,d(\kappa_1)_\mathcal{A}(\eta)+(1-\beta)\int_{\mathcal{A}^c}\eta\, d(\kappa_1)_{\mathcal{A}^c}(\eta)\\
&\rho_2=\int_{J_3(T^f)}\eta\, d\kappa_2(\eta)=\beta '\int_\mathcal{A}\eta\, d(\kappa_2)_\mathcal{A}(\eta)+
(1-\beta ')\int_{\mathcal{A}^c}\eta\, d(\kappa_2)_{\mathcal{A}^c}(\eta),
\end{split}
\end{equation}
be ergodic decompositions, where $\kappa_1,\kappa_2$ are probability
measures on $J_3(T^f)$, cf.\ Remark~\ref{rem:ergcomp}.
Let $\mathcal{A}\subset J_3(T^f)$ be the set of all 3-off-diagonal joinings and  let
\[
h:\re^2\to\mathcal{A},\quad h(t,u)=\mu^f_{t,u}.
\]
By Lemma \ref{suslin}, $\mathcal A$ is measurable and $h$ is a measurable isomorphism. Hence we can write
\[
\int_{\re^2}\mu_{t,u}^f\,dP(t,u)=\int_{\mathcal{A}}\eta \, d(h_*P)(\eta)
\text{
and }
\int_{\re^2}\mu_{t,u}^f\,d\te_*P(t,u)=\int_{\mathcal{A}}\eta\,  d(h_*\te_*P)(\eta).
\]
By \eqref{czesciowy}, \eqref{eq:ergcompp} and
by the uniqueness of
ergodic decomposition, it follows that
\[
\alpha\int_{\mathcal{A}}\eta\, d(h_*P)(\eta)+\beta(1-\alpha)\int_\mathcal{A}\eta
\,d(\kappa_1)_\mathcal{A}(\eta)=\alpha\int_{\mathcal{A}}\eta\, d(h_*\te_*P)(\eta)+\beta'(1-\alpha)\int_\mathcal{A}\eta\, d(\kappa_2)_\mathcal{A}(\eta),
\]
and hence $\beta=\beta '$. Again, by uniqueness of ergodic decomposition, we obtain that
\begin{equation*}
\alpha h_*P+\beta(1-\alpha)(\kappa_1)_\mathcal{A}=\alpha h_*\te_*P+\beta(1-\alpha)(\kappa_2)_\mathcal{A}.
\end{equation*}
Applying $(h^{-1})_*$ to this equality (recall that $h$  is a
measurable isomorphism), we conclude the proof.
\end{proof}

Let us consider
\[
\xi:\re^2\to\re,\quad \xi(x,y)=x-2y.
\]
Recall that
\[
P=\lim_{n\to \infty}(f_n',f_n)_*\mu_{W_n},
\]
where $f_n=f^{(q_n)}-a_n$ and $f_n'=f^{(2q_n)}-2a_n$. Then for every $x\in X $ we have
\begin{equation*}
\begin{split}
\xi\circ(f_n',f_n)(x)&=f_n'(x)-2f_n(x)=\sum_{i=0}^{2q_n-1}f(T^ix)-2a_n-2\Big(\sum_{i=0}^{q_n-1}f(T^ix)-a_n\Big)\\&
=\sum_{i=q_n}^{2q_n-1}f(T^ix)-\sum_{i=0}^{q_n-1}f(T^ix)=f^{(q_n)}(T^{q_n}x)-f^{(q_n)}(x).
\end{split}
\end{equation*}
It follows that
\begin{equation}\label{eq:xip}
\xi_*P=\lim_{n\to \infty}(\xi\circ(f_n',f_n))_*\mu_{W_n}=\lim_{n\to \infty}(f^{(q_n)}\circ T^{q_n}-f^{(q_n)})_*\mu_{W_n}.
\end{equation}

\begin{theorem}\label{glwynik}
Let $T:( X,\mathcal B,\mu)\to( X,\mathcal B,\mu)$ be an ergodic automorphism and let $f:X\to\re_+$ be a square integrable roof function. Suppose that there
exists a sequence $\{W_n\}_{n\in\n}$ of measurable sets, an increasing sequence $\{q_n\}_{n\in\n}$ of natural numbers, a real sequence $\{a_n\}_{n\in\n}$,
a real number $0<\al\le 1$ and a measure $P\in\mathcal P(\re^2)$ so that conditions \eqref{wl1}-\eqref{wl7} are satisfied with $q_n'=2q_n$ and $a_n'=2a_n$.
Assume that $\xi_*P=\sum_{i=0}^mc_i\de_{d_i}$ is a discrete measure with $d_0=0$, $\sum_{i=1}^mc_i>\frac{1-\al}{\al}$ and $d_i\neq-d_j$ for $i\neq j$. Then
the special flow $T^f$ is not isomorphic to its inverse.
\end{theorem}
\begin{proof}
Suppose that, contrary to our claim, $T^f$ is isomorphic to its inverse. Then, in view of Lemma~\ref{rownosc}, \eqref{glownemiary} is valid. Next note
that
\[
\xi\circ\te(t,u)=2u-t=-\xi(t,u).
\]
Hence, $ \xi_*(\te_*P)=\sum_{i=0}^mc_i\de_{-d_i}$. Applying $\xi_*$ to \eqref{glownemiary} we obtain
\[
\alpha \sum_{i=0}^mc_i\de_{d_i}+\beta(1-\alpha)\xi_*\eta_1=\alpha \sum_{i=0}^mc_i\de_{-d_i}+\beta(1-\alpha)\xi_*\eta_2
\]
with $0\leq \beta\leq 1$.
After normalization this gives
\[
\frac{\alpha}{\alpha+\beta(1-\al)} \sum_{i=0}^mc_i\de_{d_i}+\frac{\beta(1-\alpha)}{\alpha+\beta(1-\al)} \xi_*\eta_1=\frac{\alpha}{\alpha+\beta(1-\al)}
\sum_{i=0}^mc_i\de_{-d_i}+\frac{\beta(1-\alpha)}{\alpha+\beta(1-\al)} \xi_*\eta_2.
\]
The LHS probability measure has non-zero atoms at $d_i$, $i=1,\ldots,m$ and the sum of their measures is no less than
\[\frac{\alpha}{\alpha+\beta(1-\al)}\sum_{i=1}^mc_i>\frac{\alpha}{\alpha+\beta(1-\al)}\frac{1-\al}{\al}
=\frac{1-\alpha}{\alpha+\beta(1-\al)}\geq \frac{\beta(1-\alpha)}{\alpha+\beta(1-\al)}.\] The RHS probability measure has atoms at $-d_i$, $i=0,\ldots,m$
and the sum of their masses is no less than $1-\frac{\beta(1-\alpha)}{\alpha+\beta(1-\al)}$. By assumption, the union of these atoms  is disjoint from the
set of atoms $\{d_i:i=1,\ldots,m\}$. It follows that the LHS measure has at least $2m+1$ different atoms with total mass greater than $1$. This yields a
contradiction, and hence the proof is complete.
\end{proof}

\section{Interval Exchange Transformations}\label{sec:iets}
On the interval $I=[0,1)$ we consider the standard  Lebesgue measure $Leb$. Denote by $S_d$ ($d\ge 2$) the set of all permutations of $d$ elements. Let
$\Lambda^d=\{\la\in\re^d_+:|\la|=1\}$ be the standard unit simplex, where $|\la|=\sum_{i=1}^d\la_i$ is the length of a vector $\la$.  For $\la\in\re^d_+$
let $I_k=[\sum_{j<k}\la_k,\sum_{j\le k}\la_k)$ for $k=1,\ldots,d$.  For every $\la\in\re^d_+$ and $\pi\in S_d$ the {\it interval exchange transformation}
(IET) $T_{\pi,\la}:[0,|\la|)\to[0,|\la|)$ is a map that rearranges intervals $I_k$, $k=1,\ldots,d$ according to permutation $\pi$. More precisely,
$T_{\pi,\la}$ acts  on  each $I_k$ as the translation by $\sum_{ \pi(j)<k}\lambda_j-\sum_{j<k}\lambda_j$. Note that $T_{\pi,\la}$ preserves $Leb$. For
relevant background material concerning IETs, see  \cite{Veech0}, \cite{Veech1} or \cite{Viana}.

In this section we give a background material on IETs that will be necessary in Sections~\ref{sec:overites}~and~\ref{sec:pieconst} to construct a sequence
of sets $W_n$ for interval exchange transformation such that the assumptions of Theorem \ref{glwynik} are met. The constructions will be based on the
construction made by Veech in \cite{Veech3} (see also \cite{Katok}).

Let $T_{\pi,\la}:[0,1)\to [0,1)$ be an ergodic interval exchange transformation of $d$ intervals given by a permutation $\pi\in S_d$, and a length vector
$\la=(\la_1,\ldots,\la_d)\in \Lambda^d$. Then $\pi$ is irreducible, i.e.\
\[
\pi(\{1,\ldots,k\})=\{1,\ldots,k\}\text{  for some  }k\in\{1,\ldots,\,d\}\quad\Rightarrow\quad k=d.
\]
Denote by $S_d^0$  the set of all irreducible permutations in $S_d$. Let $I_j$, $j=1,\ldots, d$ be the exchanged intervals and by $\partial I_j$ we will
denote the left endpoint of the interval $I_j$, i.e.\ $\partial I_j=\sum_{i<j}\la_i$. We say that $T_{\pi,\la}$ satisfies {\it Keane's condition}, if
\[
T_{\pi,\la}^k(\partial I_i )=\partial I_j\text{ for some
}k\in\n\text{ and }i,j\in\{1\ldots \,d\}\Rightarrow k=1\text{ and
} j=1.
\]
\begin{theorem}[see \cite{Keane} and \cite{Veech2}]
If $\pi\in S_d^0$  then for a.e.\ $\la\in\Lambda^d$ the interval exchange transformation $T_{\pi,\la}$  satisfies Keane's condition and is ergodic.
\end{theorem}
From now on for any interval $J$ we will denote by $|J|$ its length. Let $s=\min\big\{|I_d|,|I_{\pi^{-1}(d)}|\big\}$. If $|I_d|\neq |I_{\pi^{-1}(d)}|$ then
we can consider the first return map to the interval $[0,1-s)$. Note that newly obtained automorphism is also an interval exchange of $d$ subintervals of
$[0,1-s)$ which is given by a pair $(\pi^1,\la^1)\in S_d^0\times\re^d_+$. The map
\[(\pi,\la)\mapsto(\pi^1(\pi,\la),\la^1(\pi,\la))=(\pi^1,\la^1)\]
is called {\it the Rauzy-Veech induction} and we will denote it by $\widehat R$. If possible we can iterate Rauzy-Veech induction. If $T_{\pi,\la}$ is an
interval exchange transformation satisfying Keane's condition then for each $n\in \n$ the $n$-th Rauzy-Veech induction is well defined.

Let $(\pi^n,\la^n)=\widehat R^n(\pi,\la)$. We will denote by $I^n$ the interval obtained after $n$-th step of Rauzy-Veech induction and by $I^n_j$  the
$j$-th exchanged intervals in the newly obtained IET for $j=1,\ldots,d$.

Denote by $A^n(\pi,\la)=A^n=[A^n_{ij}]_{i,j=1,\ldots, d}$ the $n$-th Rauzy-Veech induction matrix (or shortly Rauzy's matrix), i.e.\ $A^n_{ij}$ is the
number of times that the interval $I^n_j$ visits the interval $I_i$ under iterations of $T_{\pi,\la}$ before its first return to $I^n$. Recall that
$A^n(\pi,\la)\in SL_d(\mathbb Z)$, see \cite{Ra}.

\begin{remark}\label{rozkladnawieze}
Let $T=T_{\pi,\la}:I\to I$ be an interval exchange transformation of $d$ intervals satisfying Keane's condition. Then for every $n\in\n$ (cf.\ Lemma~4.2 in
\cite{Viana}) the interval $I$ is decomposed into $d$ Rokhlin towers of the form $\{T^iI^n_j:0\leq i<s_j^n\}$, where $s_j^n:=\sum_{i=1}^dA_{ij}^n(\pi,\la)$
for $j=1,\ldots,d$. More precisely,  $T^iI^n_j$ for $i=0,\ldots, s_j^n-1$ and $j=1,\ldots,d$ are pairwise disjoint intervals. Moreover, every interval
$T^iI^n_j$ is contained in an interval $I_k$ and $T^{s_j^n}I^n_j\subset I^n$. It follows that $T$ acts on each interval $T^iI^n_j$ as a translation and
$s_j^n$ is the first return time of interval $I_j^{n}$ to $I^{n}$ under $T$.
\end{remark}
\begin{proposition}[see \cite{Veech0}]\label{proposition}
Let $(\pi,\la)\in S_d^0\times \Lambda^d$ be such that $T_{\pi,\la}$ satisfies Keane's condition. Then
\begin{enumerate}
\item $A^n(\pi,\la)\la^n=\la$;
\item $A^n(\pi,\la)=A^1(\pi^0,\la^0)\cdot\ldots\cdot A^1(\pi^{n-1},\la^{n-1})$ with $(\pi^0,\la^0)=(\pi,\la)$;
\item there exists $n\in \n$ such that $A^n(\pi,\la)$ is strictly positive and $\pi^n=\pi$.
\end{enumerate}
\end{proposition}
For any positive $\,d\times d$ matrix $B$ let
\[
\rho(B)=\max_{1\le i,k,l\le d}\frac{B_{ij}}{B_{ik}}.
\]
Set $b_{j}=\sum_{i=1}^dB_{ij}$ and let $A$ be any nonnegative
nonsingular $\,d\times d$ matrix. The following properties are
easy to prove
\begin{equation}\label{ro1}
b_j\le\rho(B)b_k\quad\text{ for any }\quad1\le j,k\le d,
\end{equation}
\begin{equation}\label{ro2}
\rho(AB)\le\rho(B).
\end{equation}
By {\it normalized Rauzy-Veech induction} we call the map
\[R:S_d^0\times \Lambda^d\to S_d^0\times \Lambda^d,\qquad
R(\pi,\la)=\Big(\pi^1,\frac{\la^1}{|\la^1|}\Big).\] The set of permutations $S^0_d$ splits into subsets called \emph{Rauzy graphs} $G\subset S^0_d$ such
that the product $G\times \Lambda^d$ is $R$-invariant (see \cite{Viana} for details). The following theorem was proved by Veech in \cite{Veech2}.
\begin{theorem}\label{rauzy}
For every Rauzy graph $G\subset S^0_d$ there exists a $\sigma$-finite $R$-invariant measure $\zeta_G$ on $G\times\Lambda^d$ equivalent to the product of
the counting measure on $G$ and the Lebesgue measure on $\Lambda^d$ such that the normalized Rauzy-Veech induction $R$ is ergodic and recurrent on
$(G\times\Lambda^d,\zeta_G)$.
\end{theorem}
\begin{lemma}[see \cite{Ra}]
For every Rauzy graph $G\subset S^0_d$ there exists $\pi\in G$ such that $\pi(1)=d$ and $\pi(d)=1$.
\end{lemma}
\begin{remark}\label{mnozmac}
Let $T_{\pi,\la}:[0,1)\to[0,1)$ be an interval exchange transformation satisfying Keane's condition.  Then for every $\lambda'\in\Lambda^d$  we have
\[A^1(\pi,A^1(\pi,\la)\la')=A^1(\pi,\la)\text{ and }\pi^1(\pi,A^1(\pi,\la)\la')=\pi^1(\pi,\la),\] cf.\ Remark~2 in \cite{Fr}. By induction, for every $n\ge
0$ and $0\le k\le n$ we have $A^k(\pi,A^n(\pi,\la)\la')=A^k(\pi,\la)$ and the permutation obtained after $k$-th steps of the Rauzy-Veech induction from
$T_{\pi,A^n(\pi,\la)\la'}$ and $T_{\pi,\la}$ are the same. Therefore, in view of Proposition~\ref{proposition},
$\widehat{R}^n(\pi,A^n(\pi,\la)\la')=(\pi^n,\la')$, where $\pi^n$ is a permutation such that $\widehat{R}^n(\pi,\la)=(\pi^n,\la^n)$.
\end{remark}

\section{Piecewise absolutely continuous roof functions}\label{sec:abscont}
In this section we  prove Theorem~\ref{thm:abscont}, which is the main result of the paper. Let $T:=T_{\pi,\la}:[0,1)\to[0,1)$ be an ergodic IET and  let
$f:[0,1)\to\re$ be a positive piecewise absolutely continuous roof function. We will always assume that $f$ is  right-continuous. Then the derivative $Df$
is well defined almost everywhere, it is integrable and we can define the sum of jumps of $f$ as
\[
S(f):=\int_0^1Df(x)\,dx.
\]
Let us decompose the function $f$ into the sum of functions $f_{pl}$ and $f_{ac}$, where $f_{pl}$ is a piecewise linear function with the slope $S(f)$ and
$f_{ac}$ given by
\begin{equation*}
f_{ac}(x)=\int_0^xDf(t)\,dt-S(f)\,x
\end{equation*}
is an absolutely continuous function. Note that $\int_0^1Df_{ac}(t)\,dt=f_{ac}(1)-f_{ac}(0)=0$.

The proof of Theorem~\ref{thm:abscont} will be based on the following lemmas.

\begin{lemma}\label{general}
For almost every $(\pi,\la)\in S_d^0\times \Lambda^d$ and for every $r\geq 3$ there exists $D_\la\subset[0,1)^r$ of full Lebesgue measure such that:
\begin{enumerate}
\item\label{Ln0} if $f:[0,1)\to\re_+$ is piecewise linear with slope $s\neq 0$ and $f$ is continuous on  intervals exchanged by $T:=T_{\pi,\la}$, or
\item\label{L0} if $f$ is piecewise constant with no jumps of opposite value and all discontinuity points $\beta_1,\ldots,\beta_r\in[0,1)$ of $f$ are such
that $(\beta_1,\ldots,\beta_r)\in D_\la$,
\end{enumerate}
then there exists a sequence of Rokhlin towers $\{T^i\Delta_n:0\leq i<q_n\}$ for $n\in\n$ (with $q_n\to+\infty$)  such that $T^i\Delta_n$ are intervals for
all $i=0\leq i<q_n$ and $n\in\n$  and if
\begin{equation}\label{eq:jn}
J_n:=\Delta_n\cap T^{-q_n}\Delta_n\cap T^{-2q_n}\Delta_n
\end{equation}
then the sets $W_n=\bigcup_{i=0}^{q_n-1}T^{i}J_n$ satisfy conditions \eqref{wl1}-\eqref{wl7} and  $\xi_*P=\sum_{i=0}^mc_i\de_{d_i}$ is a discrete measure with $d_0=0$ such that
\begin{equation}\label{neq:condme}
\sum_{i=1}^mc_i>\frac{1-\al}{\al}\quad\text{ and }\quad d_i\neq-d_j\text{ for }i\neq j.
\end{equation}
\end{lemma}
We will present the proof of Lemma~\ref{general} under assumptions \ref{Ln0} and \ref{L0} separately. The proofs are postpone until  Sections
\ref{sec:overites} and \ref{sec:pieconst} respectively.

The method used in the proof of the following lemma was introduced in the proof of Theorem~2 in   \cite{Katok}.
\begin{lemma}\label{konst2}
Let $T:I\to I$ be an IET and let $f:I\to\re$ be a function of bounded variation. Let $\{T^i\Delta:0\leq i<q\}$ be a Rokhlin tower of intervals. Then there
exists $a\in\re$ such that
\[|f^{(q)}(x)-a|\le Var_{[0,1)}f\quad\text{ and }\quad|f^{(2q)}(x)-2a|\le 2Var_{[0,1)}f\]
for all  $x\in\bigcup_{i=0}^{q-1}T^i(\Delta\cap T^{-q}\Delta\cap T^{-2q}\Delta)$.
\end{lemma}
\begin{proof}
Let
\begin{equation*}\label{ciag2}
a:=\frac{1}{|\Delta|}\int_{\bigcup_{l=0}^{q-1}T^l\Delta}f(t)\,dt.
\end{equation*}
Then for $x\in T^k(\Delta\cap T^{-q}\Delta)$ we have
\[
\begin{split}
|f^{(q)}(x)-a|&\le\sum_{k\le i<q}\frac{1}{|\Delta|}\int_{T^i\Delta}|f(T^{i-k}x)-f(t)|\,dt+\sum_{0\le
i<k}\frac{1}{|\Delta|}\int_{T^i\Delta}|f(T^{q+i-k}x)-f(t)|\,dt\\&\le \sum_{0\le i< q}Var_{T^i\Delta}f\le Var_{[0,1]}f.
\end{split}
\]
If $x\in T^k(\Delta\cap T^{-q}\Delta\cap T^{-2q}\Delta)$ then $T^qx\in T^k(\Delta\cap T^{-q}\Delta)$. Hence,
\[
|f^{(2q)}(x)-2a|\le |f^{(q)}(x)-a|+|f^{(q)}(T^qx)-a|\le 2Var_{[0,1]}f,
\]
which concludes the proof of the lemma.
\end{proof}
The proof of the following result is partially based on the proof of Lemma~4.8 in \cite{Kul}.
\begin{lemma}\label{AClem}
Assume that $T:[0,1)\to[0,1)$ is a uniquely ergodic IET and $g:[0,1)\to\re$ is a function of bounded variation. Let $\{T^i\Delta_n:0\leq i<q_n\}$ for
$n\in\n$ be a sequence of Rokhlin towers such that $T^i\Delta_n$ are intervals for all $i=0\leq i<q_n$ and $n\in\n$ and $q_n\to+\infty$. Let
$J_n:=\Delta_n\cap T^{-q_n}\Delta_n\cap T^{-2q_n}\Delta_n$ and $W_n=\bigcup_{i=0}^{q_n-1}T^{i}J_n$. Suppose that $\lim Leb(W_n)=\al>0$.
\begin{enumerate}
\item There exist a real sequence  $\{a_n\}_{n\in\n}$ and $C>0$ such that  $|g^{(q_n)}(x)-a_n|\leq C$ and
$|g^{(2q_n)}(x)-2a_n|\leq C$ for all $x\in W_n$ and $n\in\n$.
\item There exists measure $Q\in\mathcal{P}(\re^2)$ such that
$\big(g^{(2q_n)}-2a_n,g^{(q_n)}-a_n\big)_*(Leb_{W_n})\to Q$ weakly in $\mathcal{P}(\re^2)$, up to taking a subsequence.
\item If additionally $g$ is absolutely continuous with $\int_0^1Dg(t)dt=0$ then
\[(g^{(q_n)}\circ T^{q_n}-g^{(q_n)})\chi_{W_n}\to 0\quad\text{ uniformly.}\]
\end{enumerate}
\end{lemma}
\begin{proof}
The  claim (i) follows directly from Lemma \ref{konst2}. Since the family of probability distributions
$\big(g^{(2q_n)}-2a_n,g^{(q_n)}-a_n\big)_*(Leb_{W_n})$, $n\in\n$
is supported on a compact subset of $\re^2$, there exists a limit distribution $Q$ and (ii) is proved. 

It remains to prove (iii). Assume that $g$ is absolutely continuous with $\int_0^1Dg(t)dt=0$. Then for every $\ep>0$ there exists a function $g_{\ep}\in
C^1([0,1])$ such that $Var(g_{\ep}-g)=\|Dg-Dg_{\ep}\|_{L^1}<\ep$ and $g_\ep(0)=g(0)$. Therefore,
\begin{equation}\label{Epsilon}
\Big|\int_0^1Dg_{\ep}(t)\,dt\Big|=\Big|\int_0^1(Dg_{\ep}(t)-Dg(t))\,dt\Big|\le\int_0^1|Dg_{\ep}(t)-Dg(t)|\,dt<\ep.
\end{equation}
By the unique ergodicity of $T$, we have
\[
\frac{1}{q_n}\Big|\sum_{i=0}^{q_n-1}Dg_{\ep}\circ T^i\Big|\to\Big|\int_0^1Dg_{\ep}(t)dt\Big| \quad\text{ uniformly.}\] For sufficiently large $n$ and for
$x\in W_n$ we have
\[
\begin{split}
\Big|\sum_{i=0}^{q_n-1}(g_{\ep}(T^{q_n+i}x)-g_{\ep}(T^ix))\Big|&= \Big|\sum_{i=0}^{q_n-1}\int_{T^ix}^{T^{q_n+i}x}Dg_{\ep}(y)\,dy\Big|\\&\le
\Big|\int_x^{T^{q_n}x}\Big|\sum_{i=0}^{q_n-1}Dg_{\ep}(T^it)\,dt\Big|\,dy\Big|<q_n\ep|T^{q_n}x-x|,
\end{split}
\]
where in the last inequality we used unique ergodicity and \eqref{Epsilon}.
Since $|T^{q_n}x-x|\le Leb(\Delta_n)$ and $q_n Leb(\Delta_n)=Leb (\bigcup_{i=0}^{q_n-1}T^i\Delta_n)\leq 1$, we have $q_n|T^{q_n}x-x|\le 1$, and thus
\begin{equation}\label{g1}
\Big|\sum_{i=0}^{q_n-1}(g_{\ep}(T^{q_n+i}x)-g_{\ep}(T^ix))\Big|<\ep.
\end{equation}
We adhere to the convention that $[x,y)=[y,x)$ whenever $0\leq y<x<1$. Since $[T^ix,T^{q_n+i}x)$
for $0\le i<q_n$ are included in different levels of the tower $\{T^j\Delta_n:0\leq j<q_n\}$, the intervals $[T^ix,T^{q_n+i}x)$ for $0\le i<q_n$ are pairwise disjoint.
Hence for $x\in W_n$ we get
\begin{equation}\label{g2}
\begin{split}
\Big|\sum_{i=0}^{q_n-1}&(g_{\ep}(T^{q_n+i}x)-g_{\ep}(T^ix))-\sum_{i=0}^{q_n-1}(g(T^{q_n+i}x)-g(T^ix))\Big|\\&\leq
\sum_{i=0}^{q_n-1}\Big|(g_{\ep}-g)(T^{q_n+i}x)-(g_{\ep}-g)(T^ix)\Big|\\
&\leq\sum_{i=0}^{q_n-1}Var_{[T^ix,T^{q_n+i}x)}(g_{\ep}-g)\leq Var_{[0,1)}(g_\ep-g)<\ep.
\end{split}
\end{equation}
Combining \eqref{g1} and \eqref{g2}, for sufficiently large natural  $n$ and $x\in W_n$ we have
\[
|g^{(q_n)}(T^{q_n}x)-g^{(q_n)}(x)|=\Big|\sum_{i=0}^{q_n-1}(g(T^{q_n+i}x)-g(T^ix))\Big|<2\ep,
\]
which completes the proof.
\end{proof}

\begin{proof}[Proof of Theorem~\ref{thm:abscont}]
Let $f=f_{pl}+f_{ac}$ be the decomposition of $f$ into its piecewise linear part with slope $S(f)$ and the
absolutely continuous part satisfying $\int_0^1Df_{ac}(t)dt=0$. If $f_{ac}=0$, then the assertions \ref{ACn0} and \ref{AC0} of
the theorem follow straightforwardly from Theorem~\ref{glwynik} and the parts \ref{Ln0} and \ref{L0} of Lemma \ref{general} respectively.
We will show that the result remains unchanged when $f_{ac}$ is non-zero.

Let $\{W_n\}_{n\in\n}$ be the sequence of Rokhlin towers, $\{q_n\}_{n\in\n}$ the increasing sequence of natural numbers and $\{a_n\}_{n\in\n}$ real
sequence arisen from Lemma \ref{general} applied to $f_{pl}$ and let $P\in \mathcal{P}(\re^2)$ be the weak limit of
$\big(f_{pl}^{(2q_n)}-2a_n,f_{pl}^{(q_n)}-a_n\big)_*(Leb_{W_n})$ as $n\to+\infty$. In both cases we obtain that the measure $\xi_*P\in\mathcal P(\re)$
satisfies the assumption of Theorem \ref{glwynik}. By Lemma \ref{konst2}  applied to the function $f$ and the sequence $\{W_n\}_{n\in\n}$, we obtain a
sequence $\{b_n\}_{n\in\n}$ such that
\[|f^{(q_n)}(x)-b_n|\leq Var(f)\text{ and } |f^{(2q_n)}(x)-2b_n|\leq 2Var(f)\text{ for $x\in W_n$ and $n\in\n$}\]
and there exists a weak limit $Q\in\mathcal P(\re^2)$ of measures $\big(f^{(2q_n)}-2b_n,f^{(q_n)}-b_n\big)_*(Leb_{W_n})$ for $n\in\n$. Observe that
$\xi_*Q=\xi_*P$. Indeed, by \eqref{eq:xip}, we have
\[
\xi_*Q\leftarrow\big(f^{(q_n)}\circ T^{q_n}-f^{(q_n)}\big)_*(Leb_{W_n})=\big(f_{pl}^{(q_n)}\circ T^{q_n}-f_{pl}^{(q_n)}+f_{ac}^{(q_n)}\circ
T^{q_n}-f_{ac}^{(q_n)}\big)_*(Leb_{W_n}).
\]
By Lemma~\ref{AClem},  $(f_{ac}^{(q_n)}\circ
T^{q_n}-f_{ac}^{(q_n)})\chi_{W_n}\to 0$ uniformly as $n\to\infty$.
Therefore, using Lemma~\ref{dodfun}, we have
\[
\xi_*Q=\lim_{n\to\infty}(f_{pl}^{(q_n)}\circ T^{q_n}-f_{pl}^{(q_n)})_*Leb_{W_n}=\xi_*P.
\]
It follows that the measure $Q$ also satisfies the assumption of Theorem \ref{glwynik}, which concludes the proof.
\end{proof}

\section{Non-reversibility of special flows over interval exchange transformation}\label{sec:overites}
\subsection{Piecewise linear roof functions}
In this section we  prove  Lemma \ref{general} under assumption that the roof function $f$ is piecewise linear with non-zero constant slope. To do that we
construct a sequence of towers $W_n$ for almost every interval exchange transformation $T$ so that the assumptions of Theorem \ref{glownyrach} are met. The
construction is based on the construction made by Veech in \cite{Veech3} (see also \cite{Katok}).
In the following lemma we perform the main step of construction.
\begin{lemma}\label{konst}
For every $\pi\in S^0_d$ and a.e.\ $\lambda\in\Lambda^d$ there exist a sequence $\{W_n\}_{n\in\n}$ of measurable sets and increasing sequence
$\{q_n\}_{n\in\n}$ of natural numbers so that  conditions \eqref{wl1}-\eqref{wl4} are satisfied with $q_n'=2q_n$.
\end{lemma}
\begin{proof}
Fix $(\pi_0,\la_0)\in S_d^0\times\Lambda^d$  and let $G$ be the Rauzy graph containing $\pi_0$. By Proposition~\ref{proposition}, there exists  $m\geq1$
such that $B:=A^m(\pi_0,\la_0)\in SL_d(\z)$ is a positive matrix and $\pi_0^m=\pi_0$. Let $\ep>0$ and let $\de>0$ be such that
\begin{equation}\label{def:epsdelta}
0<6\de<\ep\quad \text{ and } \quad(1-3\de)\Big(1-\rho(B)\frac{\de}{1-\de}\Big)>1-\ep.
\end{equation}
Let $Y\subset \Lambda^d$ be the set of
$\la$ such that
\begin{equation*}
\la_1>(1-\de)|\la|\ \text{ and }\ \la_j>\frac{\de}{2d}|\la|\text{ for }2\le j\le d.
\end{equation*}
Of course, $Y$ is open. Moreover, let $V=\{(\pi_0,\frac{B\la}{|B\la|}):\la\in Y\}$. Since $B\in SL_d(\mathbb Z)$, the set $V$ is also open in
$G\times\Lambda^d$. Let $\zeta_{G}$ be the measure on $G\times\Lambda^d$ obtained in Theorem \ref{rauzy}. Since $\zeta_{G}(V)>0$, by the ergodicity and
recurrence of $R$, for a.e.\ $(\pi,\la)\in G\times\Lambda^d$, there exists an increasing sequence $\{r_n\}_{n\in\n}$ such that
$(\pi^{r_n},\la^{r_n}/|\la^{r_n}|)=R^{r_n}(\pi,\la)\in V$. Therefore, $\pi_0=\pi^{r_n}$  and $\la^{r_n}\in A^m(\pi_0,\la_0)\la'$ for some $\la'\in\re^d_+$
such that $\la'/|\la'|\in Y$. In view of Remark~\ref{mnozmac}, we have
\[A^m(\pi^{r_n},\la^{r_n})=A^m(\pi_0,A^m(\pi_0,\la_0)\la')=A^m(\pi_0,\la_0)\text{ and }R^m(\pi^{r_n},\la^{r_n})=(\pi_0,\la'/|\la'|),\]
so $R^{m+r_n}(\pi,\la)\in \{\pi_0\}\times Y$. Moreover, by Proposition~\ref{proposition}, it follows that
\[
A^{m+r_n}(\pi,\la)=A^{r_n}(\pi,\la)\cdot A^m(\pi^{r_n},\la^{r_n})=A^{r_n}(\pi,\la)\cdot B.
\]
Hence, by \eqref{ro2}, we obtain that
\begin{equation}\label{ro3}
\rho(A^{m+r_n}(\pi,\la))\le\rho(B).
\end{equation}
Since $R^{m+r_n}(\pi,\la)=(\pi^{m+r_n},\frac{\la^{m+r_n}}{|\la^{m+r_n}|})\in \{\pi_0\}\times Y$, we have
\begin{equation}\label{ro4}
\la_1^{m+r_n}>(1-\de)|\la^{m+r_n}|,
\end{equation}
\begin{equation}\label{ro5}
\la_j^{m+r_n}>\frac{\de}{2d}|\la^{m+r_n}|\quad\text{  for  }\quad 2\le j\le d.
\end{equation}
Let $T:=T_{\pi,\la}$, $s_j^{n}:=\sum_{i=1}^{d}A^{m+r_n}_{ij}$ and
\[
J^n:=I^{m+r_n}_1\cap T^{-s_1^n}I^{m+r_n}_1 \cap T^{-2s_1^n}I^{m+r_n}_1.
\]
Since $T^{s_1^n}(I_1^{m+r_n})\subset I^{m+r_n}$ (see Remark~\ref{rozkladnawieze}) and $Leb$ is $T_{\pi,\la}$-invariant, we have
\[
Leb(J^n)=Leb(I^{m+r_n}_1\cap T^{s_1^n}I^{m+r_n}_1\cap T^{2s_1^n}I^{m+r_n}_1),
\]
and $T^{s_1^n}I^{m+r_n}_1\cap T^{2s_1^n}I^{m+r_n}_1\subset I^{m+r_n}$. Moreover, by \eqref{ro4},
\[
Leb(T^{s_1^n}I^{m+r_n}_1\cap T^{2s_1^n}I^{m+r_n}_1)=Leb(I^{m+r_n}_1\cap T^{s_1^n}I^{m+r_n}_1)\ge (1-2\de)|I^{m+r_n}|.
\]
Combining this with \eqref{ro4}, we have
\begin{equation}\label{szac}
Leb(J^n)=Leb(I^{m+r_n}_1\cap T^{s_1^n}I^{m+r_n}_1\cap T^{2s_1^n}I^{m+r_n}_1)\ge(1-3\de)|I^{m+r_n}|.
\end{equation}
Since $\{T^iI_1^{m+r_n}:0\leq i<s_1^n\}$ is a Rokhlin tower (see Remark~\ref{rozkladnawieze}), then so does $\{T^iJ^n:0\leq i<s_1^n\}$, that is
\begin{equation}\label{rozl}
J^n\cap T^lJ^n=\emptyset\text{  for  }1\le l<s_1^n,
\end{equation}
As  $T^{s_1^n}J^n\subset T^{s_1^n}(I_1^{m+r_n})\subset I^{m+r_n}$, by \eqref{szac} and \eqref{def:epsdelta},we have
\[
Leb(J^n\cap T^{s_1^n}J^n)>(1-6\de)|I^{m+r_n}|>(1-\ep)|I^{m+r_n}|.
\]
By Remark~\ref{rozkladnawieze}, we have that
$|\la|=\sum_{j=1}^d\la_j^{m+r_n}s^{n}_j$.  In view of \eqref{ro1},
\eqref{ro3} and \eqref{ro4}, it follows that
\[
\begin{split}
|\la|-s^n_1\la^{m+r_n}_1&=\sum_{j=2}^ds^n_j\la^{m+r_n}_j\le\rho(A^{m+r_n})s^n_1\de|\la^{m+r_n}|\le\rho(B)s^n_1\de|\la^{m+r_n}|\\&\le
\rho(B)s^n_1\frac{\de}{1-\de}\la^{m+r_n}_1< \rho(B)\frac{\de}{1-\de}|\la|.
\end{split}
\]
Therefore
\begin{equation}\label{neq:sn1}
s^n_1\la_1^{m+r_n}>\Big(1-\rho(B)\frac{\de}{1-\de}\Big)|\la|,
\end{equation}
which, by \eqref{szac}, \eqref{rozl} and \eqref{def:epsdelta}, implies that
\begin{equation}\label{lebeg}
\begin{split}
Leb(\bigcup_{l=0}^{s^n_1-1}T^lJ^n)&=s_1^nLeb(J^n)\ge (1-3\de)s_1^n|I^{m+r_n}|\\&>(1-3\de)\Big(1-\rho(B)\frac{\de}{1-\de}\Big)|\la|>(1-\ep)|\la|.
\end{split}
\end{equation}
Set
\begin{equation}\label{defwiezy}
q_n:=s^n_1\quad\text{ and }\quad
W_n:=\bigcup_{l=0}^{q_n-1}T^lJ^n.
\end{equation}
Since we assumed that $\la\in\Lambda^d$ (i.e.\ $|\lambda|=1$), by \eqref{lebeg} and by passing to a subsequence if necessary, we get
\begin{equation}\label{granica}
\lim_{n\to \infty}Leb(W_n)=\alpha\quad\text{  for some  }\quad 1\ge \alpha\ge 1-\ep.
\end{equation}
Thus condition \eqref{wl1} is satisfied.
Moreover
\begin{equation*}
Leb(W_n\triangle T^{-1}W_n)\le 2Leb(J^n)\to 0\quad\text{  as  }\quad n\to\infty,
\end{equation*}
which verifies condition \eqref{wl2}.

Note  that for every  $x\in W_n$
\begin{equation}\label{eq:nal}
\text{there exists $0\leq l<q_n$ such that $x,T^{q_n}x,T^{2q_n}x\in T^lI^{m+r_n}_1$.}
\end{equation}
Indeed, there exists $0\le l<q_n$  such that
\[x\in T^lJ^n=T^lI^{m+r_n}_1\cap T^{-q_n+l}I^{m+r_n}_1 \cap T^{-2q_n+l}I^{m+r_n}_1\]
and \eqref{eq:nal} follows.
From Remark~\ref{mnozmac}, $T^lI^{m+r_n}_1$ is an interval of length $|I^{m+r_n}_1|<|I^{m+r_n}|$.
Since $|I^{m+r_n}|\to 0$, it follows that
\begin{equation*}
\lim_{n\to\infty}\sup_{x\in W_n}|T^{q_n}x-x|=0\quad\text{ and }\quad
\lim_{n\to\infty}\sup_{x\in W_n}|T^{2q_n}x-x|=0.
\end{equation*}
By Lemma \ref{metr}, this implies \eqref{wl3} and \eqref{wl4} with $q_n'=2q_n$.
\end{proof}

\begin{lemma}\label{GRAN}
Assume that $\{W_n\}_{n\in\n}$ and $\{q_n\}_{n\in\n}$ are sequences constructed in the proof of Lemma \ref{konst}. Then for
every  $x\in W_n$ and $0\leq j<q_n$ the points $T^jx$ and
$T^{q_n+j}x$ belong to the same exchanged interval $I_i$ for some $i=1,\ldots,d$. Moreover, there exists a sequence $\{\gamma_n\}_{n\in\n}$
of positive numbers such that
\[
T^{q_n+j}x-T^jx=\g_n/q_n\text{ for all }x\in W_n,\ 0\leq
j<q_n\quad \text{and}\quad
\liminf_{n\to\infty}\gamma_n=\gamma>0.\]
\end{lemma}
\begin{proof}
Let $(\pi^{r_n+m},\la^{r_n+m})=\widehat{R}^{r_n+m}(\pi,\la)\in \{\pi_0\}\times \re_+ Y$ and let $\widetilde{T}:=T_{\pi^{r_n+m},\la^{r_n+m}}$. Then
$\widetilde{T}:I^{m+r_n}\to I^{m+r_n}$ is an IET such that $\widetilde{T}x=T^{q_n}x$ for $x\in I^{m+r_n}_1$. Moreover, since $\pi^{r_n+m}=\pi_0$, for all
$x\in I^{m+r_n}_1$ we have
\[\widetilde{T}x=x+\sum_{\pi_0(i)<\pi_0(1)}|I^{m+r_n}_i|.\]
As $\pi\in S_d^0$ and $\la^{r_n+m}/|\la^{r_n+m}|\in Y$, there exists $2\le p\le d$ such that $\pi_0(p)<\pi_0(1)$ and
\[\gamma_n:=q_n\sum_{\pi_0(i)<\pi_0(1)}|I^{m+r_n}_i|\geq q_n|I^{m+r_n}_p|\geq\frac{\de}{2d}q_n|I^{m+r_n}|.\]
The last inequality follows from \eqref{ro5}. In view of
\eqref{neq:sn1}, it follows that
\[\gamma_n\geq \frac{\de}{2d}q_n|I^{m+r_n}_1|\geq \frac{\de}{2d}\Big(1-\rho(B)\frac{\de}{1-\de}\Big)>0.\]
In summary, for every $x\in I^{m+r_n}_1$ we have $T^{q_n}x-x=\gamma_n/q_n$ and the sequence $\{\gamma_n\}_{n\in\n}$ of positive numbers is separated from zero.

Let $x\in W_n$ and $0\leq j<q_n$. By \eqref{eq:nal}, there exists $0\leq l<q_n$ such that $x,T^{q_n}x,T^{2q_n}x\in T^lI^{m+r_n}_1$. It follows that
$T^{j}x,T^{q_n+j}x\in T^kI^{m+r_n}_1$ for some $0\leq k<q_n$. Indeed, if $0\leq l+j<q_n$ then we take $k:=l+j$ and  if $q_n\leq l+j<2q_n$ then
$k:=l+j-q_n$. By Remark~\ref{rozkladnawieze}, $T^{j}x,T^{q_n+j}x\in T^kI^{m+r_n}_1\subset I_i$ for some $0\leq i<q_n$ and $T^k$ acts on the interval
$I^{m+r_n}_1$ as a translation. Therefore,
\[T^{q_n+j}x-T^{j}x=T^{q_n+j-k}x-T^{j-k}x\text{ with }T^{j-k}x\in I^{m+r_n}_1,\]
so $T^{q_n+j}x-T^{j}x=\gamma_n/q_n$. As $\liminf_{n\to\infty} \g_n>0$, this
completes the proof.
\end{proof}


\begin{proof}[Proof of  Lemma \ref{general} under assumption \ref{Ln0}]
Let $W_n$ be the tower obtained in \eqref{defwiezy} and let
$\{q_n\}_{n\in\n}$ be the  sequence of integer numbers defined in
\eqref{defwiezy}. Take $q_n'=2q_n$. Then, by Lemma \ref{konst},
conditions $\eqref{wl1}-\eqref{wl4}$ are satisfied. Moreover,
since $f$ is of bounded variation, we can apply Lemma~\ref{konst2}
to $\Delta=I_1^{m+r_n}$ and $q=q_n$. This yields a real sequence
$\{a_n\}_{n\in\n}$  and $P\in\mathcal{P}(\re^2)$ such that
\[
|f^{(q_n)}(x)-a_n|\leq Var(f)\quad\text{ and
}\quad|f^{(2q_n)}(x)-2a_n|\leq 2Var(f)\text{ for all }x\in
W_n\text{ and }n\geq \n.
\]
and
\[
P=\lim_{n\to
\infty}\big(f^{(2q_n)}-2a_n,f^{(q_n)}-a_n\big)_*(Leb_{W_n})\quad\text{
weakly in }\quad\mathcal{P}(\re^2).
\]
Thus \eqref{wl5}-\eqref{wl7} are also satisfied. By Lemma
\ref{GRAN}, for  every $x\in W_n$ and $0\leq j<q_n$ the point
$T^jx$ and $T^{q_n+j}x$ belong to an interval $I_i$ and
$T^{q_n+j}x-T^{j}x=\g_n q_n$. Since $f$ is a linear function on
$I_i$ with slope $s$, we obtain $f(T^{q_n+j}x)-f(T^{j}x)=\g_n
q_n$. Therefore, for every $x\in W_n$ we have
\[
f^{(q_n)}\circ
T^{q_n}(x)-f^{(q_n)}(x)=\sum_{j=0}^{q_n-1}(f(T^{q_n+j}x)-f(T^jx))=s\sum_{i=0}^{q_n-1}(T^{q_n+j}x-T^jx)=s\g_n\to
s\gamma
\]
with $\gamma>0$. Thus, by \eqref{eq:xip}, $ \xi_*P=\de_{s\gamma} $
with $s\gamma\neq 0$. By \eqref{granica}, we can take $\lim
Leb(W_n)=\al>0$ arbitrary close to 1. Choosing $\al>\frac{1}{2}$
guarantees that the measure $\xi_*P=\de_{s\gamma}$ satisfies
\eqref{neq:condme}, which concludes the proof.
\end{proof}

\section{Piecewise constant roof functions}\label{sec:pieconst}
In this section we prove  Lemma \ref{general} in the case  of
piecewise constant roof function.
We will need the following general lemma.
\begin{lemma}[cf.\ \cite{King}]\label{rokhlin}
Let $T$ be an ergodic automorphism of a standard probability space
$(X,\mathcal B,\mu)$.  Let $\{W_n\}_{n\in\n}$ be a sequence of
Rokhlin towers such that $\liminf_{n\to\infty}\mu(W_n)>0$ and
$\mu(J_n)\to 0$, where $J_n$ is the basis of the tower $W_n$. Then
almost every $x\in X$ we $x\in W_n$ for infinitely many $n\in\n$.
\end{lemma}
We will apply a construction and arguments similar  to those shown
in the proof of Lemma~\ref{konst} as well as some notation. As in
the proof of Lemma~\ref{konst}, the set $W_n$ arises from a
tower $\{T^i\Delta_n:0\leq i<q_n\}$ of intervals. In the proof of
Lemma~\ref{konst} we used a dominating tower derived from the
splitting of $[0,1)$ into towers (guaranteed by
Remark~\ref{rozkladnawieze}) and a specific choice of iterates for
Rauzy-Veech induction. In the present proof we will base on two
dominating towers and, roughly speaking, $\{T^i\Delta_n:0\leq
i<q_n\}$ will be a composition of these dominating towers. Next we
will construct a family of subtowers $W_n^l\subset W_n$ for
$l=1,\ldots,r$  that satisfy the assumption of
Lemma~\ref{rokhlin}. Lemma~\ref{rokhlin} ensures the existence of
a set $D\subset[0,1)^r$ of full Lebesgue measure such that for
every $(\beta_1,\ldots,\beta_r)\in D$ we have
$(\beta_1,\ldots,\beta_r)\in W_n^1\times\ldots\times W_n^r$ for
infinitely many $n\in\n$. In the final step of the proof we will
show that every choice the discontinuity points
$\beta_1,\ldots,\beta_r$ of the roof function $f$ so that
$(\beta_1,\ldots,\beta_r)\in D$ implies that  $\xi_*P$ is a
discrete measure satisfying the crucial condition
\eqref{neq:condme}.

\begin{proof}[Proof of  Lemma \ref{general} under assumption \ref{L0}]
Fix $\pi_0\in G_\pi$ such that $\pi_0(1)=d$ and $\pi_0(d)=1$ and
$\lambda_0\in \Lambda^d$ such that $(\pi_0,\la_0)$ satisfies
Keane's condition. By Theorem~\ref{proposition}, there exists
$m\in\n$ such that $B:=A^m(\pi_0,\la_0)$ has positive entries and
$\pi_0^m=\pi_0$. Let $0<\ep<\min\big(1/\rho(B)10,1/8(2r+1)\big)$
and let $\ep/3<\de'<\de<\ep/2$ be such that
\begin{equation}\label{neq:delty}
\de-\de'<\frac{\ep}{4\rho(B)}.
\end{equation}
Let $Y\subset\Lambda^d$ stand for the set of $\la\in\Lambda^d$
such that
\[\frac{1}{2}-\de<\la_1<\frac{1}{2}-\de+\frac{\de-\de'}{4},\qquad \frac{1}{2}+\de'<\la_d<\frac{1}{2}+\de'+\frac{\de-\de'}{4}.\]
Then
\[\frac{\ep}{2}<\frac{2}{3}\ep-\frac{1}{16}\ep<\de'+\de-\frac{\de-\de'}{4}<\la_d-\la_1<\de'+\de+\frac{\de-\de'}{4}<2\de.\]
Since $Y\subset\Lambda^d$ is open, by the arguments used in the
proof of Lemma \ref{konst}, for almost every $\la\in\Lambda^d$
there exists an increasing sequence of natural numbers
$\{r_n\}_{n\in\n}$ such that
$(\pi^{m+r_n},\la^{m+r_n})=\widehat{R}^{m+r_n}(\pi,\la)\in
\{\pi_0\}\times \re_+ Y$. Therefore
\begin{gather}
\la^{m+r_n}_1>\Big(\frac{1}{2}-\de\Big)|\la^{m+r_n}|,\quad\la^{m+r_n}_d>\Big(\frac{1}{2}+\de'\Big)|\la^{m+r_n}|,
\label{szacwiez}\\
\frac{\ep}{2}|\la^{m+r_n}|<\la^{m+r_n}_d-\la^{m+r_n}_1<2\de|\la^{m+r_n}|.\label{szacwiez1}
\end{gather}
Set $s_j^n=\sum_{i=1}^dA^{m+r_n}_{ij}$ for $j=1,\ldots,d$ and let
\[
J^n:=I_1^{m+r_n}\cap T^{-s_1^n-s_d^n}I_1^{m+r_n}\cap
T^{2(-s_1^n-s_d^n)}I_1^{m+r_n}.
\]
Let us consider the IET
$\widetilde{T}:=T_{\pi^{m+r_n},\la^{m+r_n}}$. Then
$\widetilde{T}=T^{s_1^n}$ on $I_1^{m+r_n}$ and
$\widetilde{T}=T^{s_d^n}$ on $I_d^{m+r_n}$. Since
$\pi^{m+r_n}=\pi_0$ and $\pi_0(1)=d$, $\pi_0(d)=1$, the interval
$I_1^{m+r_n}$ is translated by $T^{s_1^n}$ ($=\widetilde{T}$) to
an interval ending at the end of $I^{m+r_n}$ and $I_d^{m+r_n}$ is
translated by $T^{s_d^n}$ ($=\widetilde{T}$) to an interval
starting at $0$. Therefore, the first translation is by
$|I^{m+r_n}|-|I_1^{m+r_n}|$ and the second translation is by
$|I^{m+r_n}_d|-|I^{m+r_n}|$. Moreover, in view of
\eqref{szacwiez}, the interval $T^{s_1^n}I_1^{m+r_n}$ is shorter
than $I_d^{m+r_n}$, so
\begin{equation}\label{inc:s1}
T^{s_1^n}I_1^{m+r_n}\subset I_d^{m+r_n}.
\end{equation}
It follows that
\begin{equation}\label{eq:transl}
\text{$T^{s_1^n+s_d^n}$ translates $I_1^{m+r_n}$  by
$|I_d^{m+r_n}|-|I_1^{m+r_n}|=\la^{m+r_n}_d-\la^{m+r_n}_1>0$.}
\end{equation}
Therefore, $J^n$ is an interval starting from $0$ whose length is
$\la^{m+r_n}_1-2(\la^{m+r_n}_d-\la^{m+r_n}_1)$. Since
$0<2\de<\ep<1/10$, by \eqref{szacwiez1},
\begin{equation}\label{neq:rozlambdy}
\la^{m+r_n}_d-\la^{m+r_n}_1< 2\delta|I^{m+r_n}|<
\ep|I^{m+r_n}|<\frac{1}{10}|I^{m+r_n}|.
\end{equation}
In view of \eqref{szacwiez} and $\de<1/20$, it follows that
\begin{equation}\label{miarajn}
|J^n|=\la^{m+r_n}_1-2(\la^{m+r_n}_d-\la^{m+r_n}_1)>\Big(\frac{1}{2}-\de-\frac{1}{5}\Big)|I^{m+r_n}|>\frac{1}{4}|I^{m+r_n}|.
\end{equation}
By Remark~\ref{rozkladnawieze}, $\{T^i I_1^{m+r_n}:0\leq
i<s_1^n\}$ and $\{T^i I_d^{m+r_n}:0\leq i<s_1^n\}$ are Rokhlin
towers of intervals and $\sum_{j=1}^ds_j^n|I^{m+r_n}_j|=1$. Since
$T^{s_1^n}I_1^{m+r_n}\subset I_d^{m+r_n}$ (see \eqref{inc:s1}) and
$J^n\subset I_1^{m+r_n}$, we have that
\[\{T^i I_1^{m+r_n}:0\leq
i<s_1^n+s_d^n\}\quad\text{ and }\quad\{T^i J^n:0\leq
i<s_1^n+s_d^n\}\] are Rokhlin towers of intervals as well.
Let
\[q_n:=s_1^n+s_d^n\quad\text{and}\quad W_n:=\bigcup_{i=0}^{s_1^n+s_d^n-1}T^iJ^n.\]
Taking $\Delta_n:=I_1^{m+r_n}$, we have \eqref{eq:jn}.
In view of \eqref{ro1} and \eqref{ro2},
\begin{equation}\label{ro7}
s_j^n\le \rho(A^{m+r_n}(\pi,\la))s_k^n\le \rho(B)s_k^n\quad \text{ for }\quad 1\leq j,k\leq d.
\end{equation}
It follows that
\begin{equation}
\label{liminf1}
s_1^n|I^{m+r_n}|\geq\sum_{j=1}^d\frac{s_j^n}{\rho(B)}|I_j^{m+r_n}|=1/\rho(B).
\end{equation}
Thus, by \eqref{miarajn}, we have
\[
Leb(W_n)=(s_1^n+s_d^n)|J^n|>\frac{1}{4}(s_1^n+s_d^n)|I^{m+r_n}|\geq\frac{1}{4\rho(B)}.
\]
Hence, passing to a subsequence if necessary,  we have
\begin{equation}\label{limalfa}
\lim_{n\to\infty}Leb(W_n)=\alpha>0.
\end{equation}
Since$\{T^i I_1^{m+r_n}:0\leq i<q_n\}$ is a Rokhlin tower of
intervals such that $T$ acts on each its level by a translation,
by the arguments used in the proof of Lemma \ref{konst}, we have
that for every $x\in W_n$ there exists $0\leq l<q_n$ such that
$x,T^{q_n}x,T^{2q_n}x\in T^lI_1^{m+r_n}$ and for every $0\leq
i<q_n$ there exists $0\leq k<q_n$ such that $T^{i}x,T^{i+q_n}x\in
T^kI_1^{m+r_n}$.
%
It follows that
\begin{equation}\label{eq:szt1}
\sup_{x\in W_n}|T^{q_n}x-x|\leq|I^{m+r_n}|\to 0\text{ and }
\sup_{x\in W_n}|T^{2q_n}x-x|\leq|I^{m+r_n}|\to 0.
\end{equation}
Moreover, since $T^k$ acts on $I_1^{m+r_n}$ as a translation, by \eqref{eq:transl},
\begin{equation}\label{eq:latr}
T^{q_n}T^ix-T^ix=T^{q_n}T^{i-k}x-T^{i-k}x=\lambda^{m+r_n}_d-\lambda^{m+r_n}_1.\qquad(T^{i-k}x\in I_1^{m+r_n})
\end{equation}

Let us consider $r$ disjoint segments
\[J^n_l=\Big[\frac{(2l-1)|J^n|}{2r+1},\frac{2l|J^n|}{2r+1}\Big)\subset J^n\quad\text{ for
}l=1,\ldots,r.\] Set $W_n^l:=\bigcup_{i=0}^{q_n-1}T^iJ^n_l$. Since $\{T^iJ^n:0\leq i<q_n\}$ is a Rokhlin tower,
$\{T^iJ^n_l:0\leq i<q_n\}$ are also Rokhlin
towers for $1\leq l\leq r$. Note that the measure of each such tower is $\frac{1}{2r+1}$ of the measure of $W_n$. By $\eqref{limalfa}$, it follows that $\lim Leb(W_n^l)$ is positive for all $1\leq l\leq r$.
By Lemma \ref{rokhlin}, for almost every choice $(\beta_1,\ldots
\beta_r)\in[0,1)^r$ we have $\beta_l\in W^n_l$ for infinitely many $n\in\n$ for all $1\leq l\leq
r$. Consider now the sets
\[V^n_l=\bigcup_{i=0}^{q_n-1}T^{-i}\big[\beta_l-(\la^{m+r_n}_d-\la^{m+r_n}_1),\beta_l\big)\quad\text{
for }\quad l=1,\ldots,r.\] We now prove that $V^n_l$ for
$l=1,\ldots,r$ are pairwise disjoint Rokhlin towers contained in
$W_n$. Since $\beta_l\in W_n^l$, there exists $0\leq k=k(l)<q_n$
such that $T^{-k}\beta_l\in J^n_l$. Next note that, by
\eqref{neq:rozlambdy} and \eqref{miarajn},
\begin{equation}\label{szacprzes}
\la^{m+r_n}_d-\la^{m+r_n}_1<\ep|I^{m+r_n}|<\frac{1}{8(2r+1)}|I^{m+r_n}|<\frac{1}{2(2r+1)}|J^n|.
\end{equation}
Let us consider the interval
$\Big[\frac{2l-2}{2r+1}|J^n|,\frac{2l}{2r+1}|J^n|\Big)\subset
J^n$. Since $T^k$ acts on $J^n$ as a translation, the image
$T^k\Big[\frac{2l-2}{2r+1}|J^n|,\frac{2l}{2r+1}|J^n|\Big)$ is an
interval such that $\beta_l$ belongs to its right half. As the
length of this interval is $\frac{2}{2r+1}|J^n|$, by
\eqref{szacprzes}, we have
\begin{equation}\label{eq:beta1}
\big[\beta_l-2(\la^{m+r_n}_d-\la^{m+r_n}_1),\beta_l\big)\subset
T^{k}\Big[\frac{2l-2}{2r+1}|J^n|,\frac{2l}{2r+1}|J^n|\Big)\subset
W_n.
\end{equation}
Since $T^{q_n}$ acts on $W_n$ as the translation by $\la^{m+r_n}_d-\la^{m+r_n}_1$ (see \eqref{eq:latr}),
we have
\[T^{q_n}\big[\beta_l-2(\la^{m+r_n}_d-\la^{m+r_n}_1),\beta_l-(\la^{m+r_n}_d-\la^{m+r_n}_1)\big)=
\big[\beta_l-(\la^{m+r_n}_d-\la^{m+r_n}_1),\beta_l\big).\] It
follows that
\begin{equation}\label{eq:beta2}
T^{-q_n}\big[\beta_l-(\la^{m+r_n}_d-\la^{m+r_n}_1),\beta_l\big)\subset
T^{k}\Big[\frac{2l-2}{2r+1}|J^n|,\frac{2l}{2r+1}|J^n|\Big).
\end{equation}
Thus, by \eqref{eq:beta1}, for every $0\leq i \leq k$ we have
\[
T^{-i}[\beta_l-(\la^{m+r_n}_d-\la^{m+r_n}_1),\beta_l)\subset
T^{k-i}\Big[\frac{2l-2}{2r+1}|J^n|,\frac{2l}{2r+1}|J^n|\Big)\subset
W_n\] and, by \eqref{eq:beta2}, for every $k<i<q_n$ we have
\[
T^{-i}[\beta_l-(\la^{m+r_n}_d-\la^{m+r_n}_1),\beta_l)\subset
T^{q_n+k-i}\Big[\frac{2l-2}{2r+1}|J^n|,\frac{2l}{2r+1}|J^n|\Big)\subset
W_n.\] Since $\{T^jJ^n:0\leq j<q_n\}$ is a Rokhlin tower,
$\big\{T^{j}\big[\frac{2l-2}{2r+1}|J^n|,\frac{2l}{2r+1}|J^n|\big):0\leq
j<q_n\big\}$ are disjoint Rokhlin towers for $l=1,\ldots,r$. Each
interval
$T^{-i}\big[\beta_l-(\la^{m+r_n}_d-\la^{m+r_n}_1),\beta_l\big)$ is
a subset of a level of the $l$-th tower and for $0\leq i<q_n$ the
intervals are distributed in different levels of the $l$-th tower.
Therefore, $V^n_l$,  $l=1,\ldots,r$ are pairwise disjoint Rokhlin
towers all included in $W_n$. Hence
\begin{equation}\label{eq:vn}
Leb\Big(\bigcup_{l=1}^rV^n_l\Big)=r(\la^{m+r_n}_d-\la^{m+r_n}_1)(s_1^n+s_d^n).
\end{equation}
Note that the points that do not belong to $W_n$ come from three
sources, namely: the tower of height $s^n_1$ built over
$I_1^{m+r_n}\setminus J^n$, the towers of height $s^n_j$ built
over the intervals $I_j^{m+r_n}$ for $j=2,\ldots,d-1$ and the
tower of height $s^n_d$ built over $I_d^{m+r_n}\setminus
T^{s_1^n}J^n$. From \eqref{miarajn}, we have
\begin{align*}Leb(I_1^{m+r_n}\setminus
J^n)&=|I_1^{m+r_n}|-|J^n|=2(\la^{m+r_n}_d-\la^{m+r_n}_1), \\
Leb(I_d^{m+r_n}\setminus
T^{s_1^n}J^n)&=|I_d^{m+r_n}|-|J^n|=3(\la^{m+r_n}_d-\la^{m+r_n}_1)
\end{align*}
and, by \eqref{szacwiez}, the sum of lengths of intervals $I_j^{m+r_n}$, $j=2,\ldots,d-1$
is
\[|I^{m+r_n}|-\la^{m+r_n}_d-\la^{m+r_n}_1<(\de-\de')|I^{m+r_n}|.\]
It follows that
\begin{align}\label{eq:wn}
Leb([0,1)\setminus
W_n)<(\la^{m+r_n}_d-\la^{m+r_n}_1)(2s_1^n+3s_d^n)+(\de-\de')|I^{m+r_n}|
\max_{1<j<d}(s_j^n).
\end{align}
Moreover, by \eqref{eq:wn}, \eqref{eq:vn}, \eqref{szacwiez1}, \eqref{ro7}, \eqref{neq:delty} and \eqref{liminf1}, we have
\begin{align*}
Leb&\Big(\bigcup_{l=1}^rV^n_l\Big)-Leb([0,1)\setminus
W_n)\\&>(\la^{m+r_n}_d-\la^{m+r_n}_1)
\Big((r-2)s_1^n+(r-3)s_d^n\Big)-(\de-\de')\max_{1<j<d}(s_j^n)|I^{m+r_n}|\\
&>\frac{\ep}{2}s_1^n|I^{m+r_n}|-(\de-\de')\max_{1<j<d}(s_j^n)|I^{m+r_n}|\\&>\Big(\frac{\ep}{2}-
(\de-\de')\rho(B)\Big)s_1^n|I^{m+r_n}|>\frac{\ep}{4}s_1^n|I^{m+r_n}|>\frac{\ep}{4\rho(B)}.
\end{align*}
Now passing to a subsequence, if necessary, by \eqref{limalfa}, we have
\begin{equation}\label{limGamma1}
\Gamma:=\lim_{n\to\infty}Leb\Big(\bigcup_{l=1}^rV^n_l\Big)>1-\alpha.
\end{equation}

Let $f:[0,1)\to\re_+$ be a piecewise constant roof function for
which $\beta_1,\ldots,\beta_r$ are all discontinuities and with
jumps equal to $d_1,\ldots,d_r$ respectively. By assumption,
$d_j\neq-d_k$ for  $1\leq j<k\leq r$.

Let $q_n'=2q_n$. Since $\{W_n\}_{n\in\n}$ is a sequence of towers
satisfying \eqref{limalfa} and \eqref{eq:szt1}, $\{W_n\}_{n\in\n}$
meet the conditions \eqref{wl1}-\eqref{wl4}. Moreover, in view of
Lemma~\ref{konst2}, there exist $\{a_n\}_{n\in\n}$ and $P\in
\mathcal{P}(\re^2)$ that meet the conditions
\eqref{wl5}-\eqref{wl7} with $a_n'=2a_n$.
In view of \eqref{eq:xip}, we have
\[\xi_*P=\lim_{n\to \infty}(f^{(q_n)}\circ T^{q_n}-f^{(q_n)})_*(Leb_{W_n}).\]

Let $x\in W_n$ and $0\leq i<q_n$. In view of \eqref{eq:latr}, $(T^ix,T^{q_n+i}x]$ is an interval of length $\lambda^{m+r_n}_d-\lambda^{m+r_n}_1$.
Since the distances between points $\beta_l$, $l=1,\ldots,r$ are greater than $\lambda^{m+r_n}_d-\lambda^{m+r_n}_1$, the interval $(T^ix,T^{q_n+i}x]$
can contain at most one point $\beta_l$. Moreover,
\[\beta_l\in (T^ix,T^{q_n+i}x]\quad\Longleftrightarrow\quad x\in T^{-i}[\beta_l-(\lambda^{m+r_n}_d-\lambda^{m+r_n}_1),\beta_l)\]
and
\[\beta_l\in (T^ix,T^{q_n+i}x]\quad\Longrightarrow\quad f(T^{q_n+i}x)-f(T^ix)=d_l.\]
If $(T^ix,T^{q_n+i}x]$ does not contain any point $\beta_l$,
$l=1,\ldots,r$, then $f(T^{q_n+i}x)-f(T^ix)=0$. Since
$f^{(q_n)}(T^{q_n}x)-f^{(q_n)}(x)=\sum_{i=0}^{q_n-1}(f(T^{q_n+i}x)-f(T^ix))$,
by the definition of sets $V^n_l\subset W_n$,  for every  $x\in
W_n$ we have
\[
f^{(q_n)}(T^{q_n}x)-f^{(q_n)}(x)=\begin{cases}
d_l&\text{  if  } x\in V^n_l\text{  for some }l=1,\ldots,r\\
0&\text{  otherwise}.
\end{cases}
\]
Thus, by \eqref{limalfa} and \eqref{limGamma1}, it follows that
$\xi_*P$ is a discrete  measure with $r$ non-zero atoms at $d_l$
for $l=1,\ldots,r$ whose total mass is equal to
$\frac{\Gamma}{\alpha}> \frac{1-\alpha}{\alpha}$. Since
$d_j\neq-d_k$ for  $1\leq j<k\leq r$, condition \eqref{neq:condme}
is valid and the proof is complete.
\end{proof}


\end{document}